\RequirePackage[l2tabu,orthodox]{nag}
\documentclass[reqno,12pt,oneside]{amsart}
\usepackage{amsmath}
\usepackage{amssymb}
\usepackage{stmaryrd}
\usepackage{bm} 
\usepackage{xcolor}
\usepackage{mathtools}
\usepackage{enumitem}

\usepackage[nospace,noadjust]{cite}
\usepackage{geometry}

\usepackage{indentfirst}

\definecolor{linkblue}{RGB}{1,1,190}
\definecolor{citegreen}{RGB}{1,190,1}
\usepackage[linkcolor=linkblue
           ,citecolor=citegreen
           ,ocgcolorlinks        
           ,bookmarksopen=True,  
           ]{hyperref}
\usepackage{microtype}

\numberwithin{equation}{section}

\newcommand{\bN}{\mathbb{N}}
\newcommand{\bZ} {\mathbb{Z}}

\newcommand{\bQ}{\mathbb{Q}}

\newcommand{\cA}{\mathcal{A}}
\newcommand{\cS}{\mathcal{S}}

\newcommand{\cP}{\mathcal{P}}

\newcommand{\cT}{\mathcal{T}}

\DeclareMathOperator{\GL}{GL}

\DeclareMathOperator{\chr}{char}
\newcommand{\val}{\mathsf v}

\renewcommand{\vec}{\bm}
\newcommand{\ind}{\mathbf 1}
\newcommand{\defit}[1]{\textsf{#1}}

\DeclareMathOperator{\supp}{supp}

\DeclarePairedDelimiter{\card}{\lvert}{\rvert}

\newcommand{\vx}{{\vec x}}
\newcommand{\vn}{{\vec n}}

\newcommand{\vlambda}{{\vec \lambda}}
\newcommand{\valpha}{{\vec\alpha}}
\newcommand{\vbeta}{{\vec\beta}}

\newtheorem{thm}{Theorem}[section]

\newtheorem{prop}[thm]{Proposition}
\newtheorem{lem}[thm]{Lemma}
\newtheorem{cor}[thm]{Corollary}

\theoremstyle{definition}
\newtheorem{defn}[thm]{Definition}
\newtheorem{remark}[thm]{Remark}
\newtheorem{example}[thm]{Example}

\newenvironment{smallremark}{\begin{remark}\small}{\end{remark}} 

\newcommand{\llparen}{(\mkern-3mu(}
\newcommand{\rrparen}{)\mkern-3mu)}

\setlist[enumerate,1]{label=\textup{(\arabic*)}, ref=\textup{(}\arabic*\textup{)}, leftmargin=0.75cm}
\setlist[enumerate,2]{label=\textup{(}\roman*\textup{)}, ref=\textup{(}\roman*\textup{)}}

\newlist{equivenumerate}{enumerate}{1}
\setlist[equivenumerate,1]{%
  label=\textup{(\alph*)},
  ref=\textup{(}\alph*\textup{)},
  leftmargin=0.75cm
}

\newlist{equivenumerate*}{enumerate*}{1}
\setlist*[equivenumerate*,1]{%
  label=\textup{(\alph*)},
  ref=\textup{(}\alph*\textup{)},
  leftmargin=0.75cm
}

\def\binom#1#2{{#1}\choose{#2}}

\title[Arithmetic restrictions on $D$-finite series]{$D$-finite multivariate series with arithmetic restrictions on their coefficients}

\author{Jason Bell}
\address{Department of Pure Mathematics, University of Waterloo, Waterloo, ON, Canada N2L 3G1}
\email{jpbell@uwaterloo.ca}

\author{Daniel Smertnig}
\address{University of Graz, Institute for Mathematics and Scientific Computing, NAWI Graz, Heinrichstrasse 36, 8010 Graz, Austria}
\email{daniel.smertnig@uni-graz.at}

\begin{document}

\begin{abstract}
  A multivariate, formal power series over a field $K$ is a Bézivin series if all of its coefficients can be expressed as a sum of at most $r$ elements from a finitely generated subgroup $G \le K^*$; it is a Pólya series if one can take $r=1$.
  We give explicit structural descriptions of $D$-finite Bézivin series and $D$-finite Pólya series over fields of characteristic $0$, thus extending classical results of P\'olya and B\'ezivin to the multivariate setting.
\end{abstract}

\maketitle

\section{Introduction}

A fundamental result in the theory of linear recurrences due to P\'olya  \cite{polya21} asserts that if $\{f(n)\}$ is a sequence satisfying a linear recurrence and taking all of its values in the set $G\cup \{0\}$ for some finitely generated multiplicative subgroup $G$ of $\mathbb{Q}^*$, then the generating series of $f(n)$ is a finite sum of series of the form $gx^i/(1-g'x^n)$ with $g,g'\in G$ and $i,n$ natural numbers along with a finite set of monomials with coefficients in $G$; moreover, these series can be chosen to have disjoint supports.  In particular, P\'olya's theorem gives a complete characterization of sequences $\{f(n)\}\subseteq \mathbb{Q}^*$ that have the property that both $\{f(n)\}$ and $\{1/f(n)\}$ satisfy a non-trivial linear recurrence. 

P\'olya's result was later extended to number fields by Benzaghou \cite{benzaghou70} and then by Bézivin \cite{bezivin87} to all fields (even those of positive characteristic). A noncommutative multivariate version was recently proved by the authors \cite{bell-smertnig21}; the noncommutative variant can be interpreted as structural description of unambiguous weighted finite automata over a field.

The generating function of a sequence satisfying a linear recurrence is the power series in some variable $x$ of a rational function about $x=0$.  When one adopts this point of view, it is natural to ask whether such results can be extended to $D$-finite (or differentiably finite) power series.  We recall that if $K$ is a field of characteristic zero then a univariate power series $F(x) =\sum f(n)x^n\in K[[x]]$ is \defit{$D$-finite} if $F(x)$ satisfies a non-trivial homogeneous linear differential equation with polynomial coefficients:
$$\sum_{i=0}^n p_i(x) F^{(i)}(x)=0,$$ with $p_0(x),\ldots ,p_n(x)\in K[x]$. 

Univariate $D$-finite series were introduced by Stanley \cite{stanley80} and the class of $D$-finite power series is closed under many operations, such as taking $K$-linear combinations, products, sections, diagonals, Hadamard products, and derivatives, and it contains a multitude of classical generating functions arising from enumerative combinatorics \cite[Chapter 6]{stanley99}. In particular, algebraic series and their diagonals are $D$-finite, and a power series $F$ is $D$-finite if and only if its sequence of coefficients satisfies certain recursions with polynomial coefficients \cite[Theorem 3.7]{lipshitz89}.

Bézivin gave a sweeping extension of P\'olya's result, showing that a univariate $D$-finite Bézivin series over a field $K$ of characteristic $0$ with the property that there is a fixed $r\ge 1$ and a fixed finitely generated subgroup $G$ of $K^*$ such that every coefficient of $F(x)$ can be written as a sum of at most $r$ elements of $G$ is in fact a rational power series and has only simple poles \cite{bezivin86}.

A natural, and thus far unexplored, direction in which to extend the results of P\'olya and B\'ezivin is to consider multivariate analogues. A multivariate variant of $D$-finite series was given by Lipshitz \cite{lipshitz89}.  Here one has a field $K$ of characteristic zero and declares that a formal multivariate series
\[
  F=\sum_{\vec n \in \bN^d} f(\vec n) \vx^{\vec n} \in K\llbracket \vx\rrbracket=K\llbracket x_1,\ldots,x_d\rrbracket,
\]
is \defit{$D$-finite} if all the partial derivatives $(\partial/\partial x_1)^{e_1} \cdots (\partial/\partial x_d)^{e_d}F$ for $e_1$, $\ldots\,$,~$e_d \ge 0$ are contained in a finite-dimensional vector space over the rational function field $K(\vx)$.
Equivalently, for each $i \in [1,d]$, the series $F$ satisfies a linear partial differential equation of the form
\[
  P_{i,n} \cdot (\partial/\partial x_i)^n F + P_{i,n-1}  \cdot (\partial/\partial x_i)^{n-1} F + \cdots + P_{i,1} \cdot (\partial/\partial x_i) F + P_{i,0}\cdot F = 0.
\]
with polynomials $P_{i,0}$, $P_{i,1}$, $\ldots\,$,~$P_{i,n} \in K[\vx]$, at least one of which is nonzero.

In fact, many interesting classical Diophantine questions can be expressed in terms of questions about coefficients of multivariate rational power series and multivariate $D$-finite series.  To give one example, Catalan's conjecture (now a theorem due to Mihăilescu \cite{mihailescu04}) states that the only solutions to the equation $3^n=2^m+1$ are given by $(n,m)=(1,1)$ and $(2,3)$.  This is equivalent to the statement that the bivariate rational power series 
$$1/((1-3x_1)(1-x_2))-1/((1-x_1)(1-2x_2)) - 1/((1-x_1)(1-x_2))$$ has nonzero coefficients except for the coefficients of $x_1^2x_2^3$ and $x_1x_2$.  On the other hand, it is typically much more difficult to obtain results about multivariate rational functions and multivariate $D$-finite series for several reasons.  In the case of univariate rational series, the coefficients have a nice closed form and there is a strong structural description of the set of zero coefficients due to Skolem, Mahler, and Lech (see \cite[Chapter 2.1]{everest-vanderpoorten-shparlinski-ward03}).  Similarly, for $D$-finite series there are many strong results concerning the asymptotics of their coefficients (see \cite{flajolet-sedgewick09}, in particular Chapters VII.9 and VIII.7) and one can often make use of these results when considering problems in the univariate case.
Straightforward attempts at extending these approaches to higher dimensions typically fail or become too unwieldy.
For this reason new ideas are often important in obtaining multivariate extensions.

It is therefore of considerable interest to understand multivariate $D$-finite series, although doing so often presents additional technical difficulties. In this paper we consider structural results for multivariate series that are implied by the additional restrictions on the coefficients of a $D$-finite series that were imposed by B\'ezivin and P\'olya in the univariate case.
For a field $K$ of characteristic zero and a multiplicative subgroup $G \le K^*$ we set $G_0 \coloneqq G \cup \{0\}$ and write
\[
  rG_0 \coloneqq \{\, g_1 + \cdots + g_r : g_1, \ldots, g_r \in G_0 \,\}
\]
for the $r$-fold sumset of $G_0$.

In view of B\'ezivin's and P\'olya's results we give the following definitions.  
\begin{defn} \label{d:bezivin}
  A power series $F=\sum_{\vec n \in \bN^d} f(\vec n) \vx^{\vec n} \in K\llbracket \vx \rrbracket$ is
  \begin{itemize}
  \item a \defit{Bézivin series} if there exists a finitely generated subgroup $G \le K^*$ and $r \in \bN$ such that $f(\vec n) \in rG_0$ for all $\vec n \in \bN^d$;
  \item a \defit{Pólya series} if there exists a finitely generated subgroup $G \le K^*$ such that $f(\vec n) \in G_0$ for all $\vec n \in \bN^d$.
  \end{itemize}
\end{defn}

The terminology of Pólya series is standard; we have chosen to name Bézivin series based on Bézivin's results characterizing such this type of series in the univariate case \cite{bezivin86,bezivin87}.

In this paper we completely characterize multivariate $D$-finite Bézivin and Pólya series.
As an immediate consequence, we obtain a characterization of $D$-finite series whose Hadamard (sub-)inverse is also $D$-finite.
The proofs make repeated use of unit equations and results from the theory of semilinear sets.

To state our main result, we recall the notion of semilinear subsets of $\mathbb{N}^d$. A subset $\cS \subseteq \bN^d$ is \defit{simple linear} if it is of the form $\cS = \vec a_0 + \vec a_1\bN + \cdots + \vec a_s \bN$ with ($\bZ$-)linearly independent $\vec a_1$, $\ldots\,$,~$\vec a_s \in \bN^d$.
The terminology comes from the theory of semilinear sets; see Section~\ref{ssec:semilinear} below.
The following is our main result.

\begin{thm} \label{thm:main-bezivin} Let $K$ be a field of characteristic zero, let $d \ge 0$, and let
  \[
    F = \sum_{\vec n \in \bN^d} f(\vec n) \vx^{\vec n} \in K\llbracket \vec x\rrbracket=K\llbracket x_1,\ldots,x_d \rrbracket
  \]
  be a Bézivin series, with all coefficients of $F$ contained in $rG_0$ for some finitely generated subgroup $G \le K^*$ and $r \in \bN$.
  Then the following statements are equivalent.
  \begin{equivenumerate}
  \item\label{t-bezivin:dfinite} $F$ is $D$-finite.
  \item\label{t-bezivin:rational} $F$ is rational.
  \item\label{t-bezivin:skewgeom} $F$ is a \textup{(}finite\textup{)} sum of skew-geometric series with coefficients in $G$, that is, rational functions of the form
    \[
      \frac{g_0u_0}{(1-g_1 u_1)\cdots (1-g_l u_l)}
    \]
    where $u_0$, $u_1$, $\ldots\,$,~$u_l$ are monomials in $x_1$,~$\ldots\,$,~$x_d$ such that $u_1$, $\ldots\,$,~$u_l$ are algebraically independent, and $g_0$, $g_1$,~$\ldots\,$,~$g_l \in G$.

   \item\label{t-bezivin:skewgeom-refined} As in \ref{t-bezivin:skewgeom}, but in addition, the sum may be taken in such a way that for any two summands the support is either identical or disjoint.
     Moreover, every $\vec n \in \bN^d$ is contained in the support of at most $r$ summands.
   \item\label{t-bezivin:coeffs} There exists a partition of $\bN^d$ into finitely many simple linear sets so that on each such set $\cS=\vec a_0 + \vec a_1 \bN + \cdots + \vec a_s \bN$ with $\vec a_1$, $\ldots\,$,~$\vec a_s$ linearly independent,
     \[
       f(\vec a_0 + m_1 \vec a_1 + \cdots + m_s \vec a_s) = \sum_{i=1}^l g_{i,0} g_{i,1}^{m_1} \cdots g_{i,s}^{m_s} \qquad\text{for $(m_1,\ldots,m_s) \in \bN^s,$}
     \]
     where $0 \le l \le r$ and $g_{i,j} \in G$ for $i \in [1,l]$ and $j \in [0,s]$.
  \end{equivenumerate}
\end{thm}

The description of $D$-finite Pólya series, that is, the case $r=1$ of the previous theorem, deserves separate mention, although it follows readily from the more general result on Bézivin series.
Following terminology from formal language theory, a sum of power series $F_1$, $\ldots\,$,~$F_n$ is \defit{unambiguous} if the support of $F_i$ is disjoint from the support of $F_j$ for $i \ne j$ (see Definition~\ref{d:unambiguous} below).

\begin{thm} \label{thm:main-polya}
  Let $K$ be a field of characteristic zero, let $d \ge 0$, and let
  \[
    F = \sum_{\vec n \in \bN^d} f(\vec n) \vx^{\vec n} \in K\llbracket \vec x\rrbracket=K\llbracket x_1,\ldots,x_d \rrbracket
  \]
  be a Pólya series with coefficients contained in $G_0$ for some finitely generated subgroup $G \le K^*$.
  Then the following statements are equivalent.
  \begin{equivenumerate}
  \item \label{t-polya:dfinite} $F$ is $D$-finite.
  \item \label{t-polya:rational} $F$ is rational.
  \item \label{t-polya:skewgeom} $F$ is a \textup{(}finite\textup{)} \emph{unambiguous} sum of skew-geometric series with coefficients in $G$.
  \item \label{t-polya:coeff} The support of $F$ can be partitioned into finitely many simple linear sets so that on each such set $\cS=\vec a_0 + \vec a_1 \bN + \cdots + \vec a_s \bN$ with $\vec a_1$, $\ldots\,$,~$\vec a_s$ linearly independent,
     \[
       f(\vec a_0 + m_1 \vec a_1 + \cdots + m_s \vec a_s) = g_{0} g_{1}^{m_1} \cdots g_{s}^{m_s} \qquad\text{for $(m_1,\ldots,m_s) \in \bN^s,$}
     \]
     with $g_{j} \in G$ for $j \in [0,s]$.
  \end{equivenumerate}
\end{thm}

For a power series $F = \sum_{\vec n \in \bN^d} f(\vec n) \vx^{\vec n} \in K\llbracket \vec x \rrbracket$, let
\[
  F^{\dagger} = \sum_{\vec n \in \bN^d} f(\vec n)^{\dagger} \vx^{\vec n},
\]
where $f(\vec n)^{\dagger} = f(\vec n)^{-1}$ if $f(\vec n)$ is nonzero and $f(\vec n)^\dagger$ is zero otherwise.
The series $F^\dagger$ is the \defit{Hadamard sub-inverse} of $F$.

We call a power series \defit{finitary} if its coefficients are contained in a finitely generated $\bZ$-subalgebra of $K$.
The set of finitary power series is trivially closed under $K$-linear combinations, products, sections, diagonals, Hadamard products, and derivatives.
Therefore the set of finitary $D$-finite series is closed under the same operations.
Algebraic series as well as their diagonals and sections are finitary $D$-finite (see Lemma~\ref{l:algebraic-finitary}).

\begin{cor} \label{c:hadamard}
  Let $F \in K\llbracket \vx \rrbracket$ be finitary $D$-finite.
  Then $F^\dagger$ is finitary if and only if $F$ satisfies the equivalent conditions of Theorem~\ref{thm:main-polya} for some finitely generated subgroup $G \le K^*$.
  In particular, if $F$ and $F^\dagger$ are both finitary $D$-finite, then they are in fact unambiguous sums of skew-geometric series.
\end{cor}

\subsection*{Notation}
Throughout the paper, we fix a field $K$ of characteristic $0$.
When considering a Bézivin series $F \in K\llbracket \vx \rrbracket$, we will always tacitly assume that $G \le K^*$ denotes a finitely generated subgroup, and $r \ge 1$ denotes a positive integer, such that every coefficient of $F$ is contained in $rG_0$.

\section{Preliminaries}

For $a$,~$b \in \bZ$, let $[a,b] \coloneqq \{\,x \in \bZ : a \le x \le b\,\}$.
Let $\bN = \{0,1,2,\ldots \}$ and $\bN_{\ge k} = \{\, x \in \bN : x \ge k \,\}$ for $k \in \bN$.

\subsection{Semilinear sets.}
\label{ssec:semilinear}

We summarize a few results from the theory of semilinear sets, and refer to \cite[Chapter II.7.3]{sakarovitch09} and \cite{dallesandro-intrigila-varricchio12} for more details.

\begin{defn}
  Let $d \ge 1$. A subset $\cA \subseteq \bN^d$ is
  \begin{itemize}
  \item \defit{linear} if there exist $\vec a$, $\vec b_1$, $\ldots\,$,~$\vec b_l \in \bN^d$ such that $\cA = \vec a + \vec b_1 \bN + \cdots + \vec b_l \bN$;
  \item \defit{semilinear} if $\cA$ is a finite union of linear sets;
  \item \defit{simple linear} if there exist $\vec a$, $\vec b_1$, $\ldots\,$,~$\vec b_l \in \bN^d$ such that $\cA = \vec a + \vec b_1 \bN + \cdots + \vec b_l \bN$ and $\vec b_1$, $\ldots\,$,~$\vec b_l$ are linearly independent over $\bZ$.
  \end{itemize}
\end{defn}

Whenever we consider a representation of a simple linear set of the form $\vec a + \vec b_1 \bN + \cdots + \vec b_l \bN$, we shall tacitly assume that the vectors $\vec b_1$, $\ldots\,$,~$\vec b_l$ are taken to be linearly independent.

We make some observations on the uniqueness of the presentation of a linear set $\cA$.
Suppose $\cA = \vec a + \vec b_1 \bN + \cdots + \vec b_l \bN$ with $\vec a$, $\vec b_1$, $\ldots\,$,~$\vec b_l$ as above.
The element $\vec a$ is uniquely determined by $\cA$, as it is the minimum of $\cA$ in the coordinatewise partial order on $\bN^d$.
Therefore also the associated monoid $\cA - \vec a \subseteq \bN^d$ is uniquely determined by $\cA$.
The set $\{\vec b_1,\ldots, \vec b_l\}$ must contain every atom of $\cA - \vec a$, that is, every element that cannot be written as a sum of two nonzero elements of $\cA - \vec a$.
If $l$ is taken minimal, then $\{\vec b_1,\ldots,\vec b_l\}$ is equal to the set of atoms of $\cA - \vec a$, and is therefore unique.
In particular, if $\cA$ is simple linear and $\vec b_1$, $\ldots\,$,~$\vec b_l$ are linearly independent, then the representation is unique (up to order of $\vec b_1$, $\ldots\,$,~$\vec b_l$).

If $\cA$ and $\cA'$ are two linear sets with the same associated monoid, then there exist $\vec a$, $\vec a'$, $\vec b_1$, $\ldots\,$,~$\vec b_l \in \bN^d$ with $\cA = \vec a + \vec b_1 \bN + \cdots + \vec b_l \bN$ and $\cA'= \vec a' + \vec b_1 \bN + \cdots + \vec b_l \bN$.
By choosing $l$ minimal, the choice of $\vec b_1$, $\ldots\,$,~$\vec b_l$ is again unique (up to order).

The semilinear subsets of $\bN^d$ are precisely the sets definable in the Presburger arithmetic of $\bN$, by a theorem of Ginsburg and Spanier \cite{ginsburg-spanier66}.
We shall make use of the following fundamental (but non-trivial) facts.

\begin{prop} \label{p:semilinear-boolean}
  The semilinear subsets of $\bN^d$ form a boolean algebra under set-theoretic intersection and union.
  In particular, finite unions and finite intersections of semilinear sets, as well as complements of semilinear sets, are again semilinear.
\end{prop}

\begin{proof}
  By \cite[Proposition II.7.15]{sakarovitch09} a subset of $\bN^d$ is semilinear if and only if it is rational.
  By \cite[Theorem II.7.3]{sakarovitch09}, the rational subsets of $\bN^d$ form a boolean algebra.
\end{proof}

One can show that every semilinear set is a finite union of simple linear sets.
A stronger and much deeper result, that has been shown by Eilenberg and Schützenberger \cite{eilenberg-schuetzenberger69} and independently by Ito \cite{ito69}, is the following.

\begin{prop} \label{p:semilinear-disjoint}
  Every semilinear set is a finite \emph{disjoint} union of simple linear sets.
\end{prop}

\begin{proof}
  For the proof of Ito see \cite[Theorem 2]{ito69}.
  Alternatively, one may apply the more general \cite[Unambiguity Theorem]{eilenberg-schuetzenberger69} of Eilenberg and Schützenberger to the monoid $\bN^d$.
  This second proof is also contained in the book of Sakarovitch, one obtains the claim as follows:
  Let $S \subseteq \bN^d$ be semilinear.
  Then $S$ is a rational subset of $\bN^d$  (see \cite[Proposition II.7.5]{sakarovitch09} and the discussion preceeding it).
  By \cite[Theorem II.7.4]{sakarovitch09} or \cite[Unambiguity Theorem]{eilenberg-schuetzenberger69}, every rational subset of $\bN^d$ is unambiguous.
  Again by \cite[Proposition II.7.15]{sakarovitch09}, every unambiguous rational subset of $\bN^d$ is a finite disjoint union of simple linear sets.
\end{proof}

\subsection{Unit equations}

Unit equations play a central role in our proofs.
We recall the fundamental finiteness result.
For number fields it was proved independently by Evertse \cite{evertse84} and van der Poorten and Schlickewei \cite{vdpoorten-schlickewei82}; the extension to arbitrary fields appears in \cite{vdpoorten-schlickewei91}.
We refer to \cite[Chapter 6]{evertse-gyory15} or \cite[Theorem 7.4.1]{bombieri-gubler06} for more details.

Let $G$ be a finitely generated subgroup of the multiplicative subgroup $K^*$ of the field $K$.
It is important here that $\chr K=0$.
Let $m \ge 1$.
A solution $(x_1,\ldots,x_m) \in K^{m}$ to an equation of the form
\begin{equation} \label{eq:sunit}
  X_1 + \cdots + X_m = 0
\end{equation}
is \defit{non-degenerate} if $\sum_{i \in I} x_i \ne 0$ for every $\emptyset \subsetneq I \subsetneq [1,m]$.
In particular, if $m \ge 2$, then all $x_i$ of a non-degenerate solution are nonzero.

\begin{prop}[Evertse; van der Poorten and Schlickewei]
  There exist only finitely many projective points $(x_1: \cdots : x_m)$ with coordinates $x_1$, $\ldots\,$,~$x_m \in G$ such that $(x_1,\ldots,x_m)$ is a non-degenerate solution to the unit equation~\eqref{eq:sunit}.
\end{prop}

It is easily seen that there can be infinitely many degenerate solutions (even when considered as projective points), but one can recursively apply the theorem to subequations. In particular, we will commonly use an argument of the following type.

Let $\Omega$ be some index set and let $g_1$, $\ldots\,$,~$g_m \colon \Omega \to G_0$ be maps such that
\[
  g_1(\omega) + \cdots + g_m(\omega) = 0 \qquad\text{for $\omega \in \Omega$}.
\]
For every partition $\cP = \{I_1,\ldots,I_t\}$ of the set $[1,m]$, let $\Omega_\cP \subseteq \Omega$ consist of all $\omega$ such that for all $j \in [1,t]$, the tuple $(g_i(\omega))_{i \in I_j}$ is a non-degenerate solution of the unit equation $\sum_{i \in I_j} X_i = 0$.
Since every solution of a unit equation can be partitioned into non-degenerate solutions in at least one way, we obtain $\Omega = \bigcup_{\cP} \Omega_\cP$, where the union runs over all partitions of $[1,m]$.

Since there are only finitely many partitions of $[1,m]$, one can often deal with each $\Omega_\cP$ separately, or reduce to one $\Omega_\cP$ having some desirable property by an application of the pigeonhole principle. E.g., if $\Omega$ is infinite, then at least one $\Omega_\cP$ is infinite.
Similarly, if $\Omega = \bN^d$, then not all $\Omega_\cP$ can be contained in semilinear sets of rank at most $d-1$, because $\bN^d$ cannot be covered by finitely many semilinear sets of rank $\le d-1$.

\subsection{Hahn series} \label{ssec:hahn}
We recall that an abelian group $G$ is totally ordered as a group if $G$ is equipped with a total order $\le$ with the property that $a+c\le b+c$ whenever $a,b, c\in G$ are such that $a\le b$.
For the group $H=\mathbb{Q}^d$, we will give $H$ a total ordering that is compatible with the group structure by first picking positive linearly independent real numbers $\epsilon_1$ ,$\ldots\,$,~$\epsilon_d$ and then declaring that $(a_1,\ldots, a_d)\le (b_1,\ldots ,b_d)$ if and only if
\[
  \sum_{i=1}^d a_i \epsilon_i \le \sum_{i=1}^d b_i \epsilon_i.
\]
\begin{lem}\label{lem:order} The following hold for the above order:
  \begin{itemize}
  \item $\mathbb{N}^d$ is a well-ordered subset of $\mathbb{Q}^d$;
  \item if $(a_1,\ldots, a_d)< (b_1,\ldots ,b_d)$ and if $(c_1,\ldots ,c_d)\in \mathbb{Q}^d$, then there is some $N>0$ such that $n(b_1,\ldots ,b_d) > n(a_1,\ldots ,a_d)+(c_1,\ldots ,c_d)$ whenever $n\ge N$.
  \end{itemize}
\end{lem}

\begin{proof}
  Let $S$ be a non-empty subset of $\mathbb{N}^d$.
  Pick $(a_1,\ldots ,a_d)\in S$ and let $u \coloneqq \sum_{i=1}^d a_i \epsilon_i$.
  Then if $(b_1,\ldots , b_d)\in S$ is less than $(a_1,\ldots ,a_d)$, we must have $b_i \le u/\epsilon_i$ for $i \in [1,d]$.
  Thus there are only finitely many elements in $S$ that are less than $u$ and so $S$ has a smallest element.
  Next, if $(a_1,\ldots, a_d)< (b_1,\ldots ,b_d)$, then $\theta\coloneqq \sum_{i=1}^d (b_i-a_i)\epsilon_i >0$.
  Then $n\theta > \sum_{i=1}^d c_i \epsilon_i$ for all $n$ sufficiently large.
\end{proof}
The second property is equivalent to $\bQ^d$ with the given order being archimedean.

If $G$ is a totally-ordered abelian group, we can define the ring of \defit{Hahn power series}
\[
  K\llparen {\vec x}^G\rrparen \coloneqq \Big\{ \sum_{g\in G} a_g {\vec x}^g \colon a_g\in K~\textrm{for}~g\in G, ~\{g\colon a_g\neq  0\}~\textrm{is~well-ordered}\Big\}.
\]
Then $K\llparen {\vec x}^G\rrparen$ together with the obvious operations is a ring.
For
\[
  F = \sum_{g\in G} a_g {\vec x}^g \in K\llparen \vx^G \rrparen,
\]
one defines $\supp(F) \coloneqq \{\, \vx^g : g \in G, a_g \ne 0\,\}$ to be the support of $F$.
We define $[\vec x^g]F \coloneqq a_g$.
Then there is a valuation $\val \colon K\llparen x^G \rrparen \to G \cup \{\infty\}$, defined as follows: one sets $\val(F)=g$ where $\vx^g$ is the minimal monomial in the support of $F$ if $F \ne 0$, and $\val(0)=\infty$.

For $F_1 = \sum_{g\in G} a_g {\vec x}^g$ and $F_2 = \sum_{g \in G} b_g \vx^g \in K\llparen \vx^G \rrparen$, the \defit{Hadamard product} is defined as
\[
  F_1 \odot F_2 = \sum_{g \in G} a_g b_g {\vec x}^g.
\]

If $K$ is algebraically closed and $G$ is divisible, then $K\llparen {\vec x}^G\rrparen$ is an algebraically closed field.
In particular, if we use the order given above for $H=\mathbb{Q}^d$, we see that if $K$ is algebraically closed, then $K\llparen {\vec x}^H\rrparen$ is algebraically closed.
After making the identification $x_i\coloneqq {\vec x}^{\vec e_i}$, where $\vec e_i=(0,\ldots ,1,\ldots ,0)$ and where there is a $1$ in the $i$-th coordinate and zeroes in all other coordinates, we have that the formal power series ring $K\llbracket x_1,\ldots ,x_d \rrbracket$ is a subalgebra of $K\llparen x^H\rrparen$.

We will find it convenient to write ${\vec x}^{(a_1,\ldots ,a_d)} =x_1^{a_1}\cdots x_d^{a_d}$ and write $K\llparen x_1^{\mathbb{Q}},\ldots ,x_{d}^{\mathbb{Q}} \rrparen$ for $K\llparen {\vec x}^H\rrparen$.
In the other direction, we will find it convenient to abbreviate the power series ring as $K\llbracket \vec x\rrbracket = K\llbracket x_1,\ldots,x_d\rrbracket$.
These conventions are consistent with usual multi-index notation, and so for $\vec a= (a_1,\ldots,a_d)$, $\vec b = (b_1,\ldots,b_d) \in \bQ^d$,  we write $\vec a + \vec b = (a_1+b_1, \ldots, a_d + b_d)$ and $\vec a \vec b = (a_1b_1, \ldots, a_d b_d)$.
If $\vec a \in \bZ^d$ and $\vlambda=(\lambda_1,\ldots,\lambda_d)$ with $\lambda_i \in K\llparen \vec x^{H} \rrparen$, we also write $\vlambda^{\vec a}= \prod_{i=1}^d \lambda_i^{a_i}$, though we will usually only use this for $\lambda_i \in K^*$ or $\lambda_i$ a monomial.

\medskip
The set of [rational] Bézivin series is not closed under products or differentation. However, it is closed under Hadamard products and it forms a $K[\vx]$-submodule of $K\llbracket \vx \rrbracket$, as the following easy lemma shows.

\begin{lem}  \label{l:multiply}
  Let $F \in K\llbracket \vx \rrbracket$ be a Bézivin series with coefficients in $rG_0$.
  If $P \in K[\vec x]$ is a polynomial with $s$ terms in its support, then there exists a set $B \subseteq K$ of cardinality $rs$, such that
  \[
    [\vx^{\vec n}] FP \in BG_0 = \sum_{b \in B} bG_0\qquad\text{for $\vec n \in \bN^d$.}
  \]
  In particular, the series $PF$ is a Bézivin series with coefficients in $rsG_0'$ for a suitable finitely generated subgroup $G' \le K^*$.
\end{lem}

We shall have need to understand the factorization of polynomials of the form $1-c\vx^{\vec e}$ with $\vec e \in \bN^d$.

\begin{lem} \label{l:special-polys}
  Let $K$ be algebraically closed and $\vec e = (e_1,\ldots,e_d) \in \bZ^d \setminus \{\vec 0\}$.
  In the factorial domain $K[\vx^{\pm 1}]=K[x_1^{\pm 1},\ldots,x_d^{\pm 1}]$, the Laurent polynomial
  \[
    Q = 1 - c \vx^{\vec e} \qquad\text{with $c \in K^*$,}
  \]
  factors into irreducibles as
  \[
    Q = \prod_{j=1}^t (1 - \zeta^j b \vx^{\vec e/t}),
  \]
  where $t=\gcd(e_1,\ldots,e_d)$, $\zeta \in K^*$ is a primitive $t$-th root of unity, and $b^t=c$.
  In particular, the Laurent polynomial $Q$ is irreducible if and only if $\gcd(e_1,\ldots,e_d) = 1$.
\end{lem}

\begin{proof}
  The proof follows an argument of Ostrowski \cite[Theorem IX]{ostrowski76}.
  Since $(e_1/t, \ldots, e_d/t)$ is a unimodular row, there exists a matrix $A=(a_{i,j})_{i,j \in [1,d]} \in \GL_d(\bZ)$ whose first row is $(e_1/t, \ldots, e_d/t)$.
  This matrix $A$ induces a ring automorphism $\varphi$ of the Laurent polynomial ring $K[\vx^{\pm 1}]$ with $\varphi(x_i)=\prod_{j=1}^d x_j^{a_{i,j}}$.
  Then $\varphi^{-1}(Q) = 1-cx_1^{t}$.
  Since $K[x_1^{\pm 1}]$ is divisor-closed in $K[\vx^{\pm 1}]$, the problem reduces to that of factoring the univariate Laurent polynomial $1-cx_1^{t}$ in $K[x_1^{\pm 1}]$.
  Since $K[x_1^{\pm 1}]$ is obtained from the factorial domain $K[x_1]$ by inverting the prime element $x_1$, and clearly $x_1$ is not a factor of $1-cx_1^t$ in $K[x_1]$, it suffices to consider the factorization of $1-cx_1^t$ in $K[x_1]$.
  But here the result is clear.
\end{proof}

\section{Rational series with polynomial-exponential coefficients}
\label{sec:special}

In this section we consider rational series whose denominator is a product of elements of the form $1-cu$ with $c \in K^*$ and $u \in K[\vx]$ a non-constant monomial.
This will come in handy in the later sections, as it will turn out that rational Bézivin series are always of such a form.
For the class of rational series under investigation here, it is possible to give a fairly explicit structural description of their coefficient sequences, namely they are piecewise polynomial-exponential on simple linear subsets of $\bN^d$.
\begin{defn} \label{d:polyexp}
  \mbox{} Let $f \colon \bN^d \to K$ be a sequence.
  \begin{itemize}
  \item
    The sequence $f$ is \defit{piecewise polynomial-exponential on simple linear subsets of $\bN^d$} if there exists a partition of $\bN^d$ into simple linear sets $\cS_1$, $\ldots\,$,~$\cS_m$ so that for each $i \in [1,m]$, there exist $k_i \in \bN$, polynomials $A_{i,1}$, $\ldots\,$,~$A_{i,k_i} \in K[\vx]$ and $\valpha_{i,1}$, $\ldots\,$,~$\valpha_{i,k_i} \in (K^*)^d$ such that
    \[
      f(\vn) = \sum_{j = 1}^{k_i} A_{i,j}(\vec n) \valpha_{i,j}^\vn \qquad\text{for $\vn \in \cS_i$}.
    \]
  \item The sequence $f$ is \defit{piecewise polynomial on simple linear subsets of $\bN^d$} if one can moreover take $\valpha_{i,j}=(1,\ldots,1)$ for all $i \in [1,m]$ and $j \in [1,k_i]$.
   \item The sequence $f$ is \defit{piecewise exponential on simple linear subsets of $\bN^d$} if one can take the polynomials $A_{i,j}$ to be constant for all $i \in [1,m]$ and $j \in [1,k_i]$.
\end{itemize}
\end{defn}

Note that in the piecewise exponential case, sums of exponentials ($k_i > 1$) are allowed.
There is no restriction on the ranks of the simple linear sets; the rank of $\cS_i$ may be smaller than $d$, and also need not be the same for $\cS_1$, $\ldots\,$,~$\cS_m$.
Each of these three types of representation is trivially preserved under refinements of the partition.
In particular, representations of the above types are not unique.
It is not hard to see that every series $F\in K\llbracket \vec x \rrbracket$, whose coefficient sequence is polynomial-exponential on simple subsets of $\bN^d$, is rational (see Corollary~\ref{cor:pexip-is-rational} below).
We give an easy example to illustrate the definition.

\begin{example}
  Consider the sequence $f \colon \bN^2 \to \bQ$ defined by
  \[
    \begin{split}
      \sum_{m,n =0}^\infty f(m,n)x^my^n &= \frac{1}{1-2xy} + \frac{1}{1-3xy} + \frac{y}{(1-3xy)^2(1-5y)} + \frac{x}{(1-xy)(1-x)} \\
        &= \sum_{k=0}^\infty (2^k+3^k)x^ky^k + \sum_{k,l=0}^\infty (k+1)3^k5^lx^ky^{k+l+1} + \sum_{k,l=0}^\infty x^{k+l+1}y^{l}
    \end{split}
  \]
  Then
  \[
    f(m,n) =
    \begin{cases}
      2^m + 3^m & \text{if $m=n$}, \\
      \frac{1}{5} (m+1)(\frac{3}{5})^m 5^n & \text{if $m < n$,} \\
      1 & \text{if $m >n$}.
    \end{cases}
  \]
  Since
  \[
    \begin{split}
    \{\, (n,n) \in \bN^2 : n \in \bN \,\} &= (1,1)\bN,\\
    \{\, (m,n) \in \bN^2 : m < n \,\} &= (0,1) + (1,1)\bN + (0,1)\bN, \text{ and }\\
    \{\, (m,n) \in \bN^2 : m > n \,\} &= (1,0) + (1,1)\bN + (1,0)\bN,
    \end{split}
  \]
  are simple linear sets, the sequence $f$ is piecewise polynomial-exponential on simple linear subsets of $\bN^2$.
  However, $f$ is neither piecewise polynomial nor piecewise exponential on simple linear subsets of $\bN^2$.
  The coefficient sequence of $\frac{1}{1-2xy} + \frac{1}{1-3xy}$ is piecewise exponential on simple linear subsets of $\bN^2$, and the coefficient sequence of $\frac{x}{(1-xy)(1-x)}$ is piecewise polynomial on simple linear subsets of $\bN^2$.
\end{example}

The following basic properties hold.

\begin{lem} \label{l:pexp-closure}
  Let $\mathcal {PE}$, $\mathcal {P}$, $\mathcal E \subseteq K\llbracket \vx \rrbracket$ denote, respectively, the power series whose coefficient sequences are piecewise polynomial-exponential, \textup{[}polynomial, exponential\textup{]} on simple linear subsets of $\bN^d$.
  \begin{enumerate}
  \item \label{pc:module} Each of $\mathcal {PE}$, $\mathcal{P}$, and $\mathcal E$ is a $K[\vec x]$-submodule of $K\llbracket \vec x\rrbracket$ and is closed under Hadamard products.
  \item \label{pc:product} The sets $\mathcal {PE}$ and $\mathcal{P}$ are also closed under products and partial derivatives. In particular $\mathcal {PE}$ and $\mathcal {P}$ form subalgebras of $K\llbracket \vx \rrbracket$.
  \end{enumerate}
\end{lem}

\begin{proof}
  \ref{pc:module}
  Let $F$, $G$ be series that are  [polynomial-exponential, polynomial, exponential] on simple linear subsets.
  It is clear that for every monomial $u$ and every scalar $\lambda \in K$, also $\lambda uF$ is of the same form.
  Thus it suffices to show that the same is true for $F+G$.
  If $\cS$, $\cT \subseteq \bN^d$ are simple linear subsets, then the intersection $\cS \cap \cT$ is semilinear.
  By Proposition~\ref{p:semilinear-disjoint}, the intersection $\cS \cap \cT$ can be represented as a finite disjoint union of simple linear sets.
  Thus, in representations of the coefficient sequences of $F$ and $G$ as in Definition~\ref{d:polyexp}, we may assume that the simple linear sets coincide for the two series, by passing to a common refinement.
  Then the claim about $F+G$ is immediate.
  
  \ref{pc:product}
  This can again be computed on each simple linear set.
  Alternatively, it will follow from Theorem~\ref{t:polyexp} and Corollary~\ref{c:poly} below.
\end{proof}

The set $\mathcal E$ is not closed under products or derivatives, since  $(1-x)^{-1} \in \mathcal E$, but $(1-x)^{-e} \in K\llbracket x\rrbracket$ for $e \ge 2$ is polynomial-exponential but not exponential on simple linear subsets of $\bN$.
We recall some easy facts about the algebraic independence of monomials.
\begin{lem} \label{l:independence}
  Let $\vec e_1$, $\ldots\,$,~$\vec e_n \in \bN^d$, let $c_1$, $\ldots\,$,~$c_n \in K^*$, and let  $m_1$, $\ldots\,$,~$m_n \in \bN_{\ge 1}$.
  The following statements are equivalent.
  \begin{equivenumerate}
  \item\label{li:lin} The vectors $\vec e_1$, $\ldots\,$,~$\vec e_n$ are linearly independent over $\bZ$.
  \item\label{li:freeab} The monomials $\vx^{\vec e_1}$, $\ldots\,$,~$\vx^{\vec e_n}$ generate a free abelian subgroup of the unit group of $K[\vx^{\pm 1}]$.
  \item\label{li:alg} The monomials $\vx^{\vec e_1}$, $\ldots\,$,~$\vx^{\vec e_n}$ are algebraically independent over $K$.
  \item\label{li:algshiftpower} The  polynomials
    \[
      (1 - c_1 \vx^{\vec e_1})^{m_1}, \ldots, (1 - c_n \vx^{\vec e_n})^{m_n}
    \]
    are algebraically independent over $K$.
  \end{equivenumerate}
\end{lem}

\begin{proof}
  \ref{li:lin}$\,\Leftrightarrow\,$\ref{li:freeab} Clear.
  
  \ref{li:lin}$\,\Leftrightarrow\,$\ref{li:alg}
  Consider the family $({\vx}^{\vec e_1 c_1 + \cdots + \vec e_n c_n})_{c_1,\ldots,c_n \in \bN}$ of monomials.
  The monomials in this family are pairwise distinct, and hence linearly independent over $K$, if and only if $\vec e_1$, $\ldots\,$,~$\vec e_n$ are linearly independent.
  But the linear independence of $({\vx}^{\vec e_1 c_1 + \cdots + \vec e_n c_n})_{c_1,\ldots,c_n \in \bN}$ is equivalent to the algebraic independence of $\vx^{\vec e_1}, \ldots, \vx^{\vec e_n}$.
  
  \ref{li:alg}$\,\Leftrightarrow\,$\ref{li:algshiftpower}
  If  $0 \ne P \in K[y_1,\ldots,y_n]$, then $P$ vanishes on $(1 - c_1\vx^{\vec e_1}, \ldots, 1-c_n\vx^{\vec e_n})$ if and only if $P(1 - c_1 y_1, \ldots, 1 - c_n y_n)$ vanishes on $(\vx^{\vec e_1}, \ldots, \vx^{\vec e_n})$.
  Hence $\vx^{\vec e_1}$, $\ldots\,$,~$\vx^{\vec e_n}$ are algebraically independent if and only if the polynomials $1 - c_1 \vx^{\vec e_1}$, $\ldots\,$,~$1 - c_n \vx^{\vec e_n}$ are algebraically independent.

  Now, for any choice of polynomials $f_1$, $\ldots\,$,~$f_n$, the field $K(f_1,\ldots,f_n)$ is an algebraic extension of $K(f_1^{m_1},\ldots,f_n^{m_n})$, and therefore the two fields have the same transcendence degree over $K$.
  Thus $1 - c_1 \vx^{\vec e_1}$, $\ldots\,$,~$1 - c_n \vx^{\vec e_n}$ is algebraically independent if and only if $(1 - c_1 \vx^{\vec e_1})^{m_1}$, $\ldots\,$,~$(1 - c_n \vx^{\vec e_n})^{m_n}$ is algebraically independent.
\end{proof}

Looking once more at Definition~\ref{d:polyexp}, in a sequence with piecewise polynomial-exponential coefficients on simple linear sets, for each simple linear set $\cS_i$, we have $f(\vec n) = \sum_{j = 1}^{k_i} A_{i,j}(\vec n) \valpha_{i,j}^\vn$ for $\vec n \in \cS_i$.
We can represent $\cS_i$ as $\cS_i = \vec a + \vec b_1 \bN + \cdots + \vec b_s \bN$ with suitable $\vec a$, $\vec b_1$, $\ldots\,$,~$\vec b_s \in \bN^d$, where $\vec b_1$, $\ldots\,$,~$\vec b_s$ are linearly independent.
On $\cS_i$ we can therefore also consider representations of the form $f(\vec n) = \sum_{j=1}^l B_{i,j}(\vec m) \vec\beta_{i,j}^{\vec m}$ where $\vec m=(m_1,\ldots,m_s)$ with $\vec n = \vec a + m_1 \vec b_1 + \cdots + m_s \vec b_s$, and $B_{i,1}$, $\ldots\,$,~$B_{i,l} \in K[y_1,\ldots,y_s]$, $\vec\beta_{i,1}$,~$\ldots\,$,~$\vec \beta_{i,l} \in (K^*)^s$.
One easily sees how the two notions relate.

\begin{lem} \label{l:equiv-pexp}
  Let $F = \sum_{\vn \in \bN^d} f(\vn) \vx^{\vn} \in K\llbracket \vx \rrbracket$ and let $\cS = \vec a + \vec b_1 \bN + \cdots + \vec b_s \bN$ be simple linear.
  Consider the following statements.
  \begin{equivenumerate}
  \item \label{equiv-pexp:full} There exist polynomials $A_1$, $\ldots\,$,~$A_l \in K[\vec x]$ and $\valpha_1$,~$\ldots\,$,~$\valpha_l \in (K^*)^d$ such that
    \[
      f(\vec n) = \sum_{j=1}^l A_j(\vec n) \valpha_j^{\vec n} \qquad\text{for $\vn \in \cS$}.
    \]
   \item \label{equiv-pexp:subset} There exist polynomials $B_1$, $\ldots\,$,~$B_l \in K[y_1,\ldots,y_s]$ and $\vbeta_1$,~$\ldots\,$,~$\vbeta_l \in (K^*)^s$ such that
    \[
      f(\vec a + \vec b_1 m_1 + \cdots + \vec b_s m_s) = \sum_{j=1}^l B_j(\vec m) \vbeta_j^{\vec m} \qquad\text{for $\vec m=(m_1,\ldots,m_s) \in \bN^s$}.
    \]
  \end{equivenumerate}
  Then \ref{equiv-pexp:full}$\,\Rightarrow\,$\ref{equiv-pexp:subset}.
  If $K$ is algebraically closed, then also \ref{equiv-pexp:subset}$\,\Rightarrow\,$\ref{equiv-pexp:full}.
\end{lem}

\begin{proof}
  \ref{equiv-pexp:full}$\,\Rightarrow\,$\ref{equiv-pexp:subset}.
  By straightforward substitution.
  
  \ref{equiv-pexp:subset}$\,\Rightarrow\,$\ref{equiv-pexp:full}. 
  Let $\vec n=(n_1,\ldots,n_d) \in \cS$.
  Since $\vec b_1$, $\ldots\,$,~$\vec b_s$ are linearly independent, there exist unique $m_1$, $\ldots\,$,~$m_s \in \bN^d$ with $\vec n = \vec a + \vec b_1 m_1 + \cdots + \vec b_s m_s$.
  Solving this linear system over $\bQ$ there exists $N \in \bN_{\ge 1}$ and, for every $i \in[1,s]$ and $j \in [1,d]$ integers $p_i$, $q_{i,j}$ such that
  \[
    m_i = p_i/N + \sum_{j=1}^d n_j q_{i,j}/N
  \]
  for all $\vec n=(n_1,\ldots,n_d) \in \cS$ and $\vec m=(m_1,\ldots,m_s) \in \bN^s$ with $\vec n = \vec a + \vec b_1 m_1 + \cdots + \vec b_s m_s$.
  
  Let $\nu \in [1,l]$.
  Suppose $\vbeta_\nu=(\beta_{\nu,1},\ldots,\beta_{\nu,s})$ and pick $\gamma_{\nu,i} \in K^*$ with $\gamma_{\nu,i}^N = \beta_{\nu,i}$ for $i \in [1,s]$.
  Set $\valpha_\nu=(\alpha_{\nu,1}, \ldots\, \alpha_{\nu,d})$ with
  \[
    \alpha_{\nu,j} \coloneqq \prod_{i=1}^s \gamma_{\nu,i}^{q_{i,j}} \qquad \text{for $j \in [1,d]$}
  \]
  and
  \[
    A_\nu(x_1,\ldots,x_d) \coloneqq B_\nu\Big(p_1/N + \sum_{j=1}^d x_j q_{1,j}/N,\, \ldots,\, p_s/N + \sum_{j=1}^d x_jq_{s,j}/N \Big) \prod_{i=1}^s \gamma_{\nu,i}^{p_i}.
  \]

  If $\vec n = \vec a + \vec b_1 m_1 + \cdots + \vec b_s m_s$, then
  \[
    \begin{split}
      A_\nu(\vec n) \valpha_\nu^{\vec n}
      &= B_\nu(\vec m) \prod_{i=1}^s \gamma_{\nu,i}^{ p_i} \prod_{j=1}^d \alpha_{\nu,j}^{n_j}
      = B_\nu(\vec m) \prod_{i=1}^s \gamma_{\nu,i}^{p_i} \cdot \prod_{j=1}^d \prod_{i=1}^s \gamma_{\nu,i}^{q_{i,j}n_j} \\
      & = B_\nu(\vec m) \prod_{i=1}^s \Big( \gamma_{\nu,i}^{ p_i} \prod_{j=1}^d \gamma_{\nu,i}^{q_{i,j}n_j} \Big)
      = B_\nu(\vec m) \prod_{i=1}^s \gamma_{\nu,i}^{m_i N} \\
      & = B_\nu(\vec m) \prod_{i=1}^s \beta_{\nu,i}^{m_i} = B_\nu(\vec m) \vec\beta_\nu^{\vec m}. \qedhere
    \end{split}
  \]
\end{proof}

The algebraic closure of $K$ was only used to guarantee the existence of $N$-th roots of the $\beta_{\nu,i}$ in the proof of \ref{equiv-pexp:subset}$\,\Rightarrow\,$\ref{equiv-pexp:full}.
Hence, the condition can clearly be weakened to the existence of suitable roots.
For instance, if $K=\mathbb R$ and all $\beta_{\nu,j}$ are positive, the implication \ref{equiv-pexp:subset}$\,\Rightarrow\,$\ref{equiv-pexp:full} holds.
However, the condition cannot be entirely removed, as the following example shows.

\begin{example} \label{exm:algclos}
  Let 
  \[
    F = \frac{1}{1-2x^2} = \sum_{n=0}^\infty 2^n x^{2n} \in \mathbb Q\llbracket x \rrbracket.
  \]
  Then $f(1+2m) = 0$ and $f(2m)=2^m$ for all $m \in \bN$.
  On $2\bN$ we thus have a representation of the form \ref{equiv-pexp:subset}.

  However, suppose that for all $n \in 2\bN$, there is a representation $\sqrt{2}^n = \sum_{i=1}^l A_i(n) \alpha_i^n$ with polynomials $A_1$, $\ldots\,$,~$A_l \in \overline\bQ[x]$ and pairwise distinct $\alpha_1$, $\ldots\,$,~$\alpha_l \in \overline\bQ$.
  Then $2^m = \sum_{i=1}^l A_i(2m) \alpha_i^{2m}$ for all $m \in \bN$.
  Since such a representation by an exponential polynomial is unique (see, e.g., \cite[Corollary 6.2.2]{berstel-reutenauer11}), we must have $\alpha_i^2=2$, hence without restriction $l=2$ and $\alpha_1=\sqrt{2}$, $\alpha_2=-\sqrt{2}$.
  Therefore $\alpha_i \not\in \bQ$.
  Of course, over $\bQ(\sqrt{2})$ one has $f(n) = \tfrac{1}{2} \sqrt{2}^n + \tfrac{1}{2} (-\sqrt{2})^n$ for all $n \in \bN$.
\end{example}

In a representation as in \ref{equiv-pexp:full} of Lemma~\ref{l:equiv-pexp}, it is easily seen that the polynomials are not uniquely determined: for instance, to represent the coefficient sequence $f$ of $\frac{1}{1-xy}$ on $\cS=(1,1)\bN$ as $f(n,n)=A(n,n)$, we can take any polynomial $A \in 1 + (x-y)K[x,y]$.
However, we now show that the situation is better in a representation as in \ref{equiv-pexp:subset}.
The proof is essentially the same as the standard one in the univariate case (see, e.g., \cite[Chapter 6.2]{berstel-reutenauer11}).

\begin{prop} \label{p:pexp-unique}
  Let $F = \sum_{\vn \in \bN^d} f(\vn) \vx^{\vn} \in K\llbracket \vx \rrbracket$ and let $\cS = \vec a + \vec b_1 \bN + \cdots + \vec b_s \bN$ be a simple linear set.
  Let $B_1$, $\ldots\,$,~$B_l \in K[y_1,\ldots,y_s]$ and $\vbeta_1$,~$\ldots\,$,~$\vbeta_l \in (K^*)^s$ where the vectors $\vbeta_j$ are pairwise distinct such that
  \[
    f(\vec a + \vec b_1 m_1 + \cdots + \vec b_s m_s) = \sum_{j=1}^l B_j(\vec m) \vbeta_j^{\vec m} \qquad\text{for $\vec m=(m_1,\ldots,m_s) \in \bN^s$}.
  \]
  Then the polynomials $B_j$ and the coefficient vectors $\vbeta_j$ are uniquely determined \textup{(}up to order\textup{)}.
\end{prop}

\begin{proof}
  It suffices to consider $\cS=\bN^d$.
  Let $\Lambda \subseteq K^*$ be a finite set and $e \in \bN$.
  Let $A_{e} \subseteq K\llbracket \vec x\rrbracket$ be the $K$-vector space of rational series spanned by rational functions of the form
  \[
    \frac{1}{(1-\lambda_1x_1)^{k_1}\cdots (1-\lambda_dx_d)^{k_d}}
  \]
  with $\lambda_1$, $\ldots\,$,~$\lambda_d \in \Lambda$ and $k_1$, $\ldots\,$,~$k_d \in [1,e]$.
  Then $\dim A_e = \card{\Lambda}^d e^d$.
  
  Now, in the group algebra $K[\vec x][ \langle \Lambda \rangle^d]$, let $V_e$ be the $K$-vector space consisting of elements of the form
  \[
    \sum_{j=1}^l C_j(\vec x)\vec\gamma_j,
  \]
  with $\vec\gamma_j \in \Lambda^d$ and polynomials $C_j \in K[\vec x]$ with $\deg_{x_i}(C_j) \le e-1$ for all $i \in [1,d]$ and $j \in [1,l]$.
  Then  $\dim V_{\vec e} \le \card{\Lambda^d}e^d$.

  For $i \in [1,d]$ let $\lambda_i \in \Lambda$ and $k_i \in [1,e]$.
  Since
  \[
    \prod_{i=1}^d (1 - \lambda_i x_i)^{-k_i} = \sum_{n_1,\ldots,n_d=0}^\infty {\binom{k_1+n_1-1}{n_1} } \cdots {\binom{k_d+n_d-1}{n_d} }\lambda_1^{n_1} \cdots \lambda_d^{n_d} x_1^{n_1}\cdots x_d^{n_d},
  \]
  the vector space $V_e$ maps surjectively onto $A_e$.
  Since $\dim V_e \le \dim A_e$, this is a bijection. This implies the claim.
\end{proof}

We also record the following consequence.

\begin{cor} \label{c:poly-exp-repr}
  Let $u_0$, $u_1$, $\ldots\,$,~$u_s \in K[\vec x]$ be monomials with $\vec u\coloneqq (u_1,\ldots,u_s)$ algebraically independent and let $C \subseteq K^*$ be a finite set. For $\vec\alpha=(\alpha_1,\ldots,\alpha_s) \in C^s$ and $\vec e = (e_1,\ldots,e_s) \in \bN_{\ge 1}$, let \[
    Q_{\vec \alpha, \vec e} \coloneqq (1-\alpha_1u_1)^{e_1}\cdots (1-\alpha_su_s)^{e_s} \in K[\vec u].
  \]
  Then, for $F \in K\llbracket \vx \rrbracket$, the following statements are equivalent.
  \begin{enumerate}
  \item
    For each $\vec\alpha \in C^s$ and $\vec e=(e_1,\ldots,e_s) \in \bN_{\ge 1}^s$, there exist $\lambda_{\vec\alpha,\vec e} \in K$, all but finitely many of which are zero, such that
    \[
      F = u_0 \sum_{\vec\alpha \in C^s} \sum_{\vec e \in \bN_{\ge 1}^s} \frac{\lambda_{\vec\alpha,\vec e}}{Q_{\vec\alpha,\vec e}}.
    \]
  \item
    For each $\valpha \in C^s$, there exists a polynomial $A_{\valpha} \in K[\vec u]$, such that
    \[
      F = u_0\sum_{\vec n \in \bN^s} \sum_{\vec \alpha \in C^s} A_{\vec \alpha}(\vec n) \valpha^{\vec n} {\vec u}^{\vn}.
    \]
  \end{enumerate}
  Moreover, the polynomials $A_\valpha$ are uniquely determined by $F$.
  If $\valpha \in C^s$ and 
  \[
    m_i = \max\{\, e_i : \vec e = (e_1,\ldots,e_i, \ldots, e_s) \in \bN^s \text{ with } \lambda_{\vec \alpha,\vec e} \ne 0 \,\}
  \]
  \textup{(}taking $m_i=-\infty$ if the set is empty\textup{)}, then $\deg_{u_i}(A_{\valpha}) = m_i - 1$.
  In particular, $A_{\vec \alpha}$ is constant if and only if $\lambda_{\vec \alpha,\vec e}=0$ whenever $e_i>1$  for some $i
  \in [1,s]$.
\end{cor}

For the next step we use a multivariate partial fraction decomposition.
The result is due to Leĭnartas \cite{leinartas78}, an easy to read exposition of the (short) proof in English is given by Raichev in \cite{raichev12}.

\begin{prop}[Leĭnartas's decompositon] \label{p:pfd}
  Let $F=P/Q \in K(\vx)$ with $P$,~$Q \in K[\vec x]$ and $Q \ne 0$, and let $Q=Q_1^{e_1} \cdots Q_t^{e_t}$ with $Q_1$, $\ldots\,$,~$Q_t \in K[\vec x]$ irreducible and pairwise non-associated and $e_1$, $\ldots\,$,~$e_t \ge 1$.
  Then
  \[
    F = \sum_{\vec b=(b_1,\ldots,b_t) \in S} \frac{P_{\vec b}}{Q_1^{b_1}\cdots Q_t^{b_t}},
  \]
  for a finite set $S \subseteq \bN^d$ and polynomials $P_{\vec b} \in K[\vx]$, and for every $\vec b =(b_1, \ldots, b_t) \in S$, the polynomials $\{\, Q_i : b_i > 0 \,\}$ appearing in the denominator are algebraically independent and have a common root in $\overline{K}$.
\end{prop}

Unlike the univariate case, in this representation, unfortunately the exponents $b_i$ can exceed $e_i$, see Example~\ref{e:pdf-exp} below.

We can now characterize power series whose coefficient sequences are polynomial-exponential on simple linear subsets of $\bN^d$.

\begin{thm} \label{t:polyexp}
  Let $K$ be algebraically closed.
  The following statements are equivalent for a power series $F = \sum_{\vn \in \bN^d} f(\vn) \vx^{\vn} \in K\llbracket \vx \rrbracket$.
  \begin{equivenumerate}
  \item\label{t-polyexp:denom} One has
    \[
      F = \frac{P}{(1-c_1u_1)\cdots (1-c_lu_l)}
    \]
    with $c_1$, $\ldots\,$,~$c_l \in K^*$, with non-constant monomials $u_1$, $\ldots\,$,~$u_l \in K[\vx]$, and $P \in K[\vx]$.
  \item\label{t-polyexp:polyexp}
    The sequence $f$ is piecewise polynomial-exponential on simple linear subsets of $\bN^d$.
  \end{equivenumerate}
\end{thm}

\begin{proof}
  \ref{t-polyexp:denom} $\Rightarrow$ \ref{t-polyexp:polyexp}
  Lemma~\ref{l:special-polys} allows us to assume that the polynomials $1-c_{i}u_{i}$ for $i \in [1,l]$ are all irreducible.
  Gathering associated polynomials, and resetting the notation, we may write
  \[
    F = \frac{P}{(1-c_1u_1)^{e_1} \cdots (1-c_lu_l)^{e_l}},
  \]
  with irreducible polynomials $1 - c_{i} u_{i}$ that are pairwise coprime and $e_1$, $\ldots\,$,~$e_l \ge 1$.

  By Proposition~\ref{p:pfd} we may now express
  \[
    F = \sum_{I \subseteq [1,l]} \frac{P_I}{Q_I},
  \]
  with $P_I \in K[\vx]$ and $Q_I = \prod_{i \in I} (1-c_iu_i)^{b_{I,i}}$ where $b_{I,i} \ge 1$ for all $I \subseteq [1,l]$ and $i \in I$, and the different irreducible factors of $Q_I$ are algebraically independent.
  By Lemma~\ref{l:independence}, the monomials $\{\, u_i : i \in I \,\}$ are algebraically independent for $I \subseteq [1,l]$.
  We can therefore apply Corollary~\ref{c:poly-exp-repr} followed by \ref{equiv-pexp:subset}$\,\Rightarrow\,$\ref{equiv-pexp:full} of Lemma~\ref{l:equiv-pexp}, to deduce that every $1/Q_I$ is piecewise polynomial-exponential on simple linear subsets of $\bN^d$.
  By \ref{pc:module} of Lemma~\ref{l:pexp-closure}, the same is true for $F$.

  \ref{t-polyexp:polyexp} $\Rightarrow$ \ref{t-polyexp:denom}
  By assumption, there exists a partition of $\bN^d$ into simple linear sets, such that on each simple linear set $\cS$, the coefficients of $F$ can be represented as in \ref{equiv-pexp:full} of Lemma~\ref{l:equiv-pexp}.
  Applying \ref{equiv-pexp:full}$\,\Rightarrow\,$\ref{equiv-pexp:subset} of Lemma~\ref{l:equiv-pexp}, followed by Corollary~\ref{c:poly-exp-repr}, we see that $F$ is a $K$-linear combination of rational functions of the form
  \[
    \frac{u_0}{(1-c_1u_1)\cdots (1-c_su_s)}
  \]
  with $c_1$, $\ldots\,$,~$c_s \in K^*$ and monomials $u_0$, $u_1$, $\ldots\,$,~$u_s \in K[\vx]$.
\end{proof}

\begin{example} \label{e:pdf-exp}
  The exponents in the multivariate partial fraction decomposition can increase.
  For instance,
  \[
    F = \frac{1}{(1-x)(1-y)(1-xy)} = \frac{x}{(1-x)(1-xy)^2} + \frac{1}{(1-y)(1-xy)^2}.
  \]
  So, despite the fact that the denominator of $F$ has no repeated factors, the coefficient polynomials in a polynomial-exponential decomposition may be non-constant.
  In the example, we get $[x^ny^n]F = (n+1)$ for $n \in \bN$.
\end{example}

The following is a corollary of the trivial direction of Theorem~\ref{t:polyexp} (and hence could have been observed before).

\begin{cor} \label{cor:pexip-is-rational}
  If the coefficient sequence of $F \in K \llbracket \vx \rrbracket$ is piecewise polynomial-exponential on simple linear subsets of $\bN^d$, then $F$ is rational.
\end{cor}

Looking at the proof of Theorem~\ref{t:polyexp} again, we see that the special case in which every constant in the denominator is a root of unity corresponds to the coefficient sequence being piecewise polynomial on simple linear subsets of $\bN^d$.
Hence we recover the following (well-known) result.

\begin{cor} \label{c:poly}
  Let $K$ be algebraically closed.
  The following statements are equivalent for a power series $F = \sum_{\vn \in \bN^d} f(\vn) \vx^{\vn} \in K\llbracket \vx \rrbracket$.
  \begin{equivenumerate}
  \item\label{t-poly:denom} One has
    \[
      F = \frac{P}{(1-\omega_1u_1)\cdots (1-\omega_lu_l)}
    \]
    with $\omega_1$, $\ldots\,$,~$\omega_l \in K^*$ roots of unity and with non-constant monomials $u_1$, $\ldots\,$,~$u_l \in K[\vx]$, and $P \in K[\vx]$.
  \item\label{t-poly:poly}
    The sequence $f$ is piecewise polynomial on simple linear subsets of $\bN^d$.
  \end{equivenumerate}
\end{cor}

\begin{smallremark}
  Generating series of the form
  \[
    \frac{1}{(1-u_1)\cdots(1-u_l)}
  \]
  with monomials $u_1$, $\ldots\,$,~$u_l$ play a role in several areas of mathematics.
  They count the number of solutions to a Diophantine linear system over the natural numbers \cite{dahmen-micchelli88}, and appear in combinatorics \cite[Chapter 4.6]{stanley12}, representation theory \cite{heckman82}, and commutative algebra (as Hilbert series) \cite{stanley96,miller-sturmfels05}.

  The coefficient sequences of such series are called \defit{vector partition functions}.
  After some earlier work by Blakley \cite{blakley64} and by Dahmen and Micchelli \cite{dahmen-micchelli88}, Sturmfels \cite{sturmfels95} gave a description of their structure in terms of the chamber complex of the matrix whose columns are the exponents of the monomials $u_1$, $\ldots\,$,~$u_l$.
  
  From a structural point of view, the set $\bN^d$ is partitioned into finitely many polyhedral cones with apex at the origin, and on each such cone the vector partition function $f$ is given by a multivariate quasi-polynomial (however, the results in \cite{sturmfels95} are much more specific than this description).
  This type of decomposition is equivalent to $f$ being polynomial on simple linear subsets of $\bN^d$ --- a good overview of the connection with the theory of semi-linear sets is given by D’Alessandro, Intrigila, and Varricchio \cite{dallesandro-intrigala-varricchio12} \footnote{The definition of a simple linear set in the paper contains a typo; only $\vec b_1$, $\ldots\,$,~$\vec b_n$ are supposed to be linearly independent.}, the specific result can be found in \cite[Proposition 2 and Corollary 1]{dallesandro-intrigala-varricchio12}.

  Proofs of this decomposition (e.g. in \cite{sturmfels95}) typically involve results on counting lattice points in polytopes.
  While we only derived the structural form of the decomposition here, we only needed purely algebraic methods.
  Neither the generalization to the polynomial-exponential case nor the use of purely algebraic methods appear to be entirely new: positive real coefficients where also permitted, for instance, by Brion and Vergne \cite{brion-vergne97}.
  A purely algebraic derivation of the polynomial case is given by Fields in \cite{fields00,fields02}, which we have discovered through O'Neill \cite[Theorem 2.10]{oneill17}.
  The result on the partial fraction decomposition is not used by Fields, but similar reductions are used for the denominators.
  The partial fraction decomposition in connection with multivariate rational generating series is used by Mishna in \cite[Chapter 92]{mishna20}.
  However, no structural decomposition result is derived there.
  
  Since we could not locate a reference that derives the structural result for the the polynomial-exponential case over an arbitrary field, we have chosen to include a proof, although we assume this is at least ``essentially known''.
\end{smallremark}

\subsection*{The exponential case}
We now look at the other extremal case, where the coefficients are piecewise exponential on simple linear subsets of $\bN^d$.

\begin{defn} \label{d:skew-geom} Let $F \in K\llbracket \mathbf x \rrbracket$.
  \begin{enumerate}
  \item
    The series $F$ is \defit{skew-geometric} if there exist $c_0 \in K$, $c_1$, $\ldots\,$,~$c_l \in K^*$ and monomials $u_0$, $u_1$, $\ldots\,$, $u_l$ such that $u_1$, $\ldots\,$,~$u_l$ are algebraically independent and
    \[
      F = \frac{c_0 u_0}{(1-c_1u_1)\cdots (1-c_l u_l)}.
    \]
  \item
    The series $F$ is \defit{geometric} if moreover $\{u_1,\ldots,u_l\} \subseteq \{x_1,\ldots,x_d\}$.
  \end{enumerate}
\end{defn}

For a skew-geometric series $F$ as above, the exponents of the monomials in the support of $F$ form a simple linear set.
The coefficient of $u_0 u_1^{e_1} \cdots u_l^{e_l}$ in $F$  is $c_0 c_1^{e_1}\cdots c_l^{e_l}$.
Every skew-geometric series is rational by definition.
Moreover its coefficient series is exponential on simple linear subsets of $\bN^d$.

\begin{defn} \label{d:unambiguous}
  Let $F_1$, $\ldots\,$~$F_n \in K\llbracket \vx \rrbracket$.
  \begin{enumerate}
  \item The sum $F_1 + \cdots + F_n$ is \defit{unambiguous} if the sets $\supp(F_1)$, $\ldots\,$,~$\supp(F_n)$ are pairwise disjoint.
  \item The sum $F_1 + \cdots + F_n$ is \defit{trivially ambiguous} if for all $i$,~$j \in [1,n]$
    \[
      \supp(F_i) \cap \supp(F_j) = \emptyset \qquad\text{or}\qquad \supp(F_i)=\supp(F_j).
    \]
  \end{enumerate}
\end{defn}

\begin{example}
  For a set $\cS \subseteq \bN^d$ let
  \[
    \ind_\cS \coloneqq \sum_{\vn \in \cS} \vx^{\vn} \in K\llbracket\vx\rrbracket.
  \]
  If $\cS = \vec a + \vec b_1 \bN + \cdots \vec b_l \bN$ is a simple linear set, then
  \[
    \ind_{\cS} = \frac{\vx^{\vec a}}{(1-\vx^{\vec b_1}) \cdots (1 - \vx^{\vec b_l})}
  \]
  is skew-geometric. If $\cS$ is semilinear, then it is a finite disjoint union of simple linear sets, and therefore $\ind_\cS$ is an unambiguous sum of skew-geometric series.
\end{example}

\begin{lem} \label{l:skewgeom-decomposition}
  \begin{enumerate}
  \item \label{lsg:restriction} If $F \in K\llbracket \vx \rrbracket$ is skew-geometric, and $\cS \subseteq \bN^d$ is a simple linear set with $\{\, \vx^\vn : \vn \in \cS \,\} \subseteq \supp(F)$, then $F \odot \ind_\cS$ is skew-geometric.
  \item \label{lsg:poly-exp}
    If a series $F \in K \llbracket \vx \rrbracket$ has a coefficient sequence that is piecewise exponential on simple linear subsets of $\bN^d$, then $F$ is a sum of skew-geometric series.
    If $K$ is algebraically closed, the converse holds.
  \item \label{lsg:refinement} Every sum of skew-geometric series can be expressed as a trivially ambiguous sum of skew-geometric series. 
  \end{enumerate}
\end{lem}

\begin{proof}
  \ref{lsg:restriction}
  Suppose
  \[
    F = \frac{c_0 \vx^{\vec a}}{(1-c_1 \vx^{\vec b_1}) \cdots ( 1 - c_l \vx^{\vec b_l})},
  \]
  with $(\vec b_1,\ldots, \vec b_l) \in \bN^d$ linearly independent, $\vec a \in \bN^d$, and $c_0$, $c_1$, $\ldots\,$,~$c_l \in K^*$ (if $c_0=0$ there is nothing to show). Let $\cS_0 = \vec a + \vec b_1 \bN + \cdots + \vec b_l \bN$ and suppose $\cS = \vec p + \vec q_1 \bN + \cdots + \vec q_s \bN$ with $\cS \subseteq \cS_0$.

  Let $\vec p = \vec a + \mu_1 \vec b_1 + \cdots + \mu_l \vec b_l$ with $\mu_1$, $\ldots\,$,~$\mu_l \in \bN$.
  Since $\vec p + \vec q_i \in \cS_0$ for all $i \in [1,s]$, we have $\vec q_i = t_{i,1} \vec b_1 + \cdots + t_{i,l} \vec b_l$ with $t_{i,j} \in \bZ$.
  Since also $\vec p + n \vec q_i \in \cS_0$ for all $n \in \bN$, we must have $t_{i,j} \ge 0$ for all $i \in [1,s]$ and $j \in [1,l]$.
  Now if $\vec n= \vec p + m_1 \vec q_1 + \cdots + m_s \vec q_s \in \cS$, then
  \[
    [\vx^\vn]F = c_0 \prod_{j=1}^l c_j^{\mu_j + \sum_{i=1}^s m_i t_{i,j}} = c_0 \prod_{j=1}^l c_j^{\mu_j} \cdot \prod_{i=1}^s \Big( \prod_{j=1}^l c_j^{t_{i,j}} \Big)^{m_i},
  \]
  and so
  \[
    F \odot 1_\cS = \frac{d_0 \vx^{\vec p}}{(1 - d_1 \vx^{\vec q_1})\cdots (1- d_s \vx^{\vec q_s})},
  \]
  for suitable $d_0$, $d_1$, $\ldots\,$,~$d_s \in K^*$. (The crucial observation in this straightforward proof was $\vec q_i \in \vec b_1 \bN + \cdots + \vec b_l \bN$.)

  \ref{lsg:poly-exp}
  Suppose $F$ has a coefficient sequence that is piecewise exponential on simple linear subsets.
  We apply the direction \ref{equiv-pexp:full}$\,\Rightarrow\,$\ref{equiv-pexp:subset} of Lemma~\ref{l:equiv-pexp}, and observe that since the polynomials $A_j$ are constant, so are the $B_j$.
  Then it is immediate that $F$ is a sum of skew-geometric series.
  For the converse direction, we analogously use \ref{equiv-pexp:subset}$\,\Rightarrow\,$\ref{equiv-pexp:full} of Lemma~\ref{l:equiv-pexp}.

  \ref{lsg:refinement} 
  The semilinear sets form a boolean algebra, and every semilinear set is a finite disjoint union of simple linear sets.
  Given any sum of skew-geometric series, we can use \ref{lsg:restriction} to refine their support in such a way that the sum can be represented as a trivially ambiguous one.
\end{proof}

\section{Rationality of \texorpdfstring{$D$}{D}-finite Bézivin series}

In this section we prove the following theorem.

\begin{thm} \label{thm:rational}
  If $F \in K\llbracket \vx \rrbracket$ is $D$-finite and Bézivin, then $F$ is rational.
\end{thm}

To do so, we first recall the notion of a $P$-recursive sequence, introduced by Lipshitz \cite{lipshitz89}, and prove a lemma that encapsulates, and extends to a multivariate setting, the crucial part of Bézivin's argument.
This lemma and its consequences will prove useful on several occasions.

Let $f \colon \bN^d \to K$ be a $d$-dimensional sequence.
A \defit{$k$-section} of $f$ is a sequence (of dimension $<d$) obtained by fixing some of the coordinates of $f$ at values $\le k$, where $k \in \bN$.
Formally, a $k$-section is a sequence $g \colon \bN^{J} \to K$ with $J \subsetneq [1,d]$ and $c_i \in [0,k-1]$ for every $i \in [1,d]\setminus J$, such that
\[
  g( (n_j)_{j\in J} ) = f(n_1,\ldots,n_d)  \quad\text{with}\quad n_i=c_i \quad\text{for}\quad i \in [1,d]\setminus J.
\]
 
\begin{defn} \label{def:precursive}
  A sequence $f\colon \bN^d \to K$ is \defit{$P$-recursive of size $k \ge 0$} if the following two conditions are satisfied.
  \begin{enumerate}[label=(\roman*)]
  \item for every $j\in[1,d]$ and every $\vec a \in [0,k]^d$ there exist polynomials $Q_{j,\vec a} \in K[y]$ such that
    \begin{equation}\label{eq:precursive}
      \sum_{\vec a \in [0,k]^d} Q_{j,\vec a}(n_j) f(\vec n - \vec a) = 0 \qquad\text{for all}\qquad \vec n=(n_1,\ldots,n_d) \in \bN_{\ge k}^d,
    \end{equation}
    and for every $j \in [1,d]$ there exists at least one $\vec a \in [0,k]^d$ with $Q_{j,\vec a} \ne 0$.
  \item if $d > 1$, then all $k$-sections of $f(\vec n)$ are $P$-recursive.
  \end{enumerate}
  The sequence $f$ is \defit{$P$-recursive} if it is $P$-recursive of size $k$ for some $k \ge 0$.
\end{defn}

We remark that the sum in the recursion runs over a $d$-dimensional hypercube, but the coefficient polynomials are univariate.

By a theorem of Lipshitz, a power series is $D$-finite if and only if its sequence of coefficients is $P$-recursive \cite[Theorem 3.7]{lipshitz89}.
While this is not completely obvious, it is possible to use these recursions and finitely many initial values to compute arbitrary values of $f$ \cite[\S 3.10]{lipshitz89}.
In particular, if $f$ is $P$-recursive, then there exists a finitely generated subfield $K_0 \subseteq K$ such that $f(\vec n) \in K_0$ for all $\vec n \in \bN^d$.

Every section of a $P$-recursive sequence, that is, a subsequence obtained by fixing some of the arguments, is again $P$-recursive \cite[Theorem 3.8(iv)]{lipshitz89}.

We need the following easy observation.

\begin{lem} \label{l:zero}
  Let $f\colon \bN^d \to K$ be $P$-recursive of size $k$, let $j \in [1,d]$, and let $Q_{j,\vec a} \in K[y]$ be polynomials, not all zero, such that
  \[
    \sum_{\vec a \in [0,k]^d} Q_{j,\vec a}(n_j) f(\vec n - \vec a) = 0 \qquad\text{for all}\qquad \vec n=(n_1,\ldots,n_d) \in \bN_{\ge k}^d.
  \]
  Further, let $c \in \bN$ be sufficiently large so that $Q_{j,\vec a}(c') \ne 0$ whenever $c' \ge c$ and $Q_{j,\vec a} \ne 0$.
  If there exist $l_1$, $\ldots\,$,~$l_d \ge c$ such that $f(\vec n)$ vanishes on
  \[
    \bigcup_{i=1}^d \bN_{\ge c}^{i-1} \times [l_i,l_i+k-1] \times \bN_{\ge c}^{d-i},
  \]
  then $f(\vec n)$ vanishes on $\bN_{\ge l_1} \times \cdots \times \bN_{\ge l_d}$.
\end{lem}

\begin{proof}
  Fix a well-order on $\bN^d$ for which $(\bN^d,+)$ is an ordered semigroup (for instance, a lexicographical order).
  We proceed by contradiction and assume that $\vec m =(m_1,\ldots,m_d) \in \bN_{\ge l_1} \times \cdots \times \bN_{\ge l_d}$ is minimal with $f(\vec m)\ne 0$.
  Then $m_i \ge l_i+k$ for all $i \in [1,d]$.
  Let $\vec b \in [0,k]^d$ be the minimum, with respect to the well-order, with $Q_{j,\vec b} \ne 0$.
  Taking $\vec n = \vec m + \vec b$, we see that Equation~\eqref{eq:precursive} allows us to express $f(\vec m)$ as a linear combination of certain $f(\vec m')$ with $\vec m' < \vec m$ and $\vec m' \in \bN_{\ge l_1} \times \cdots \times \bN_{\ge l_d}$, showing $f(\vec m)=0$,  a contradiction.
\end{proof}

The following proof is a straightforward, if somewhat technical, adaption of Bézivin's argument \cite{bezivin86} to cover multi-dimensional sequences.
The restriction to a finitely generated field ensures the applicability of the following lemma.
\begin{lem}
  If $K$ is a finitely generated as a field, then
  \[
    \sqrt{G}\coloneqq \{\, x \in K^* : x^n \in G \text{ for some $n \ge 1$\} \,}
  \]
  is a finitely generated group.
\end{lem}

\begin{proof}
  By assumption $G$ is a finitely generated group.
  Therefore there exists a finitely generated $\bZ$-subalgebra $A \subseteq K$ containing $G$.
  By a theorem of Nagata, the integral closure $\overline{A}$ of $A$ is a finitely generated $A$-module (see \cite[Theorem (31.H) in Chapter 12]{matsumura80}).
  Therefore $\overline{A}$ is a finitely generated $\bZ$-algebra as well.
  A theorem of Roquette (see \cite[Corollary 7.5 of Chapter 2]{lang83}) implies that the group of units $(\overline A)^\times$ is finitely generated.
  Since $\sqrt{G} \subseteq (\overline A)^{\times}$, the claim follows.
\end{proof}

A subgroup $G \le K^*$ is \defit{root-closed} (in $K^*$) if $\sqrt{G}=G$.
Note that always $\sqrt{G} \subsetneq K^*$.
We can even find $l \in (\bN_{\ge 1} \setminus \sqrt{G}) \subseteq K^*$, say, by choosing for $l$ a suitable prime number.

\begin{lem} \label{l:geom}
  Suppose that the field $K$ is finitely generated.
  Let $G \le K^*$ be a finitely generated subgroup and $G_0 \coloneqq G \cup \{0\}$.
  
  Let $\Omega$ be a set, let $n \ge 0$, and let $f_1$,~$\ldots\,$,~$f_n \colon \Omega \to G_0$ and $\pi \colon \Omega \to K$ be maps such that there exist polynomials $Q_1$, $\ldots\,$,~$Q_n \in K[y]$ with
  \[
    \sum_{j = 1}^n Q_j(\pi(\omega)) f_j(\omega) = 0 \qquad \text{for all $\omega \in \Omega$}.
  \]
  Then, for all finite subsets $B \subseteq K$, there exists a finitely generated root-closed group $G' \supseteq G$, such that the following holds:
  for all $l \in K^* \setminus G'$, there exists $e_0 \in \bN$ such that
  \begin{equation}\label{eq:bzvn0}
    \sum_{j = 1}^n Q_j(0) f_j(\omega) = 0
  \end{equation}
  for all $\omega \in \Omega$ with $\pi(\omega) \in C$, where 
  \[
    C = \big\{\, l^{e} + b : e \in \bN_{\ge e_0} \text{ and } b \in B \,\big\} \subseteq K,
  \]
\end{lem}

\begin{proof}
  Let $G'$ be a finitely generated group containing $G$ and all coefficients of the polynomials $Q_j(y+b)$ where $b \in B$.
  We may moreover assume $\sqrt{G'}=G'$.
  Enlarging $G$ if necessary, we even assume $G'=G=\sqrt{G}$ for notational simplicity.

  If $l \in K^*\setminus G$, then in particular $l$ is not a root of unity and thus for every $b \in K$, the set $\{\, l^e + b : e \in \bN \,\}$ is infinite.
  Moreover, if $l$ and $b$ are fixed, then the exponent $e$ is uniquely determined by $l^e +b$.

  We proceed by induction on $\card{B}$.
  For $B = \emptyset$, there is nothing to show because then $C=\emptyset$.
  So fix $b \in B$ such that $B=B' \uplus \{b\}$ with $\card{B'} < \card{B}$.
  Applying the induction hypothesis, there exists $e_0'\in \bN$ such that \eqref{eq:bzvn0} holds for all $\omega \in \Omega$ with $\pi(\omega) \in C'$, where
  \[
    C' \coloneqq \{\, l^{e} + b' : e \in \bN_{\ge e_0'},\, b' \in B' \,\}.
  \]

  Let
  \[
    \Omega_{b} \coloneqq \{\, \omega \in \Omega : \pi(\omega) = l^e + b \text{ with } e \in \bN \,\}.
  \]
  For $\omega \in \Omega_{b}$ with $\pi(\omega) = l^{e} + b$, define $\varepsilon(\omega) \coloneqq e$.
  For $j \in [1,n]$, let
  \[
    Q_j(y+b) = \sum_{s=0}^m q_{j,s} y^s \qquad\text{with}\qquad q_{j,s} \in K.
  \]
  Then
  \[
    0 = \sum_{j=1}^n Q_j(\pi(\omega)) f_j(\omega) = \sum_{j=1}^n \sum_{s=0}^m q_{j,s} l^{\varepsilon(\omega)s} f_j(\omega)
  \]
  may be considered as a solution to the unit equation $X_1 + \cdots + X_{n(m+1)}=0$ over the group $\langle G,l\rangle$.
  Let $I = [1,n] \times [0,m]$.
  For every partition $\cP=\{I_1,\ldots, I_p\}$ of $I$, let $\Omega_\cP \subseteq \Omega_{b}$ be the set satisfying: for all $\nu \in [1,p]$,
  \begin{itemize}
  \item $\sum_{(j,s) \in I_{\nu}} q_{j,s} l^{\varepsilon(\omega) s} f_j(\omega) = 0$, and
  \item $\sum_{(j,s) \in I} q_{j,s} l^{\varepsilon(\omega) s} f_j(\omega) \ne 0$ for all $\emptyset \ne I \subsetneq I_\nu$.
  \end{itemize}
  Since the sets $\Omega_\cP$, with $\cP$ ranging over all partitions, cover $\Omega_{b}$, it is sufficient to establish the claim of the lemma for each $\Omega_\cP$ separately.
  So fix $\cP$.
  If $\varepsilon(\Omega_{\cP})$ is finite, the claim is trivially true by choosing $e_0 \ge e_0'$ sufficiently large.
  We may therefore assume that $\varepsilon(\Omega_\cP)$ is infinite.

  The crucial step lies in showing that in the subsum indexed by $I_\nu$, all the monomials that occur must be equal.
  Thus, explicitly, we show: if $(j_1,s_1)$, $(j_2,s_2) \in I_\nu$ for some $\nu \in [1,p]$, then $s_{1}=s_{2}$.
  Clearly we only have to consider the case where $\card{I_\nu} \ge 2$.
  Then $q_{j_1,s_1} l^{\varepsilon(\omega)s_1} f_{j_1}(\omega) \ne 0$ and $q_{j_2,s_2} l^{\varepsilon(\omega) s_2} f_{j_2}(\omega) \ne 0$ for all $\omega \in \Omega_\cP$.
  By construction,
  \[
    \sum_{(j,s) \in I_{\nu}} q_{j,s} l^{\varepsilon(\omega) s} f_j(\omega) = 0
  \]
  is a non-degenerate solution to the unit equation $X_1 + \cdots + X_{\card{I_\nu}} = 0$ over $\langle G,l \rangle$.
  Thus there is a finite set $Y$ with
  \[
    l^{\varepsilon(\omega) (s_1 - s_2)}  \underbrace{q_{j_1,s_1} q_{j_2,s_2}^{-1} f_{j_1}(\omega) f_{j_2}(\omega)^{-1}}_{\in G} \in Y.
  \]
  for all $\omega \in \Omega_\cP$.
  Then $l^{\varepsilon(\omega) (s_1 - s_2)} \in \bigcup_{y \in Y} yG$ for all $\omega \in \Omega_\cP$.
  
  Since $\varepsilon(\omega)$ takes infinitely many values, there exist $y \in Y$ and $\omega$, $\omega' \in \Omega_\cP$ with $\varepsilon(\omega) \ne \varepsilon(\omega')$, and
  \[
    l^{\varepsilon(\omega)(s_1 - s_2)},\ l^{\varepsilon(\omega')(s_1 - s_2)}  \in yG.
  \]
  But then $l^{(\varepsilon(\omega) - \varepsilon(\omega'))(s_1 - s_2)} \in G$.
  By choice of $l$, this implies $s_{1}=s_{2}$.

  Taking the union over all $I_\nu$ containing a fixed $s \in [0,m]$ in the second coordinate, we conclude
  \[
    l^{\varepsilon(\omega) s} \sum_{j=1}^n q_{j,s}  f_j(\omega) = 0,
  \]
  and therefore $\sum_{j=1}^n q_{j,s} f_j(\omega)=0$ for all $s \in [0,m]$.
  But then
  \[
    0 = \sum_{j=1}^n \sum_{s=0}^m q_{j,s} {y}^{s} f_j(\omega) = \sum_{j=1}^n Q_j(y+b) f_j(\omega) 
  \]
  for all $\omega \in \Omega_\cP$.
  Substituting $y=-b$ into this polynomial, we obtain
  \[
    \sum_{j=1}^n Q_j(0) f_j(\omega) = 0 \qquad\text{for all}\qquad{\omega \in \Omega_\cP}. \qedhere
  \]
\end{proof}

\begin{remark}
  Since $K$ has characteristic $0$, the group $K^*$ is never finitely generated.
  In particular, one can always find $l$ as required in the previous theorem.
\end{remark}

As a first easy consequence we obtain the following.

\begin{lem} \label{l:poly-constant}
  Let $P \in K[y]$ be a polynomial, and suppose there exists $r \in \bN$ such that $P(n) \in rG_0$ for all sufficiently large $n \in \bN$.
  Then $P$ is constant.
\end{lem}

\begin{proof}
  Working in a subfield, we may assume that $K$ is a finitely generated field.
  Let $P = \sum_{i=0}^t c_i y^i$ with $t \in \bN$ and $c_0$, $\ldots\,$,~$c_t \in K$.
  Fix $n_0 \in \bN$ such that $P(n) \in rG_0$ for $n \ge n_0$, and let $g_1$, $\ldots\,$,~$g_r \colon \bN \to G_0$ be such that
  \[
    P(n) - g_1(n) - \cdots - g_r(n) = 0 \qquad\text{for $n \ge n_0$.}
  \]
  We will apply a simple special case of Lemma~\ref{l:geom} to this equation. 
  To do so, we choose $\Omega=\bN$ with $\pi\colon \bN \to K$ given by $\pi(n)=n$, $Q_1=P$, $f_1=1$, and $f_j=g_{j-1}$, $Q_j=-1$ for $j \in [2,r+1]$. 
  For the set $B$ we choose $B=\{0\}$.
  By Lemma~\ref{l:geom} there exists $l \in \bN_{\ge 1} \setminus \sqrt{G}$ and $e_0 \in \bN$ such that, for all $e\ge e_0$,
  \[
    P(0) = g_1(l^e) + \cdots + g_r(l^e).
  \]
  This implies $P(0)=P(l^e)$ for $e \ge e_0$.
  Since a non-constant polynomial can take a fixed value at most $\deg(P)-1$ times, it follows that $P$ is constant.
\end{proof}

This immediately applies to series with polynomial-exponential coefficients as follows.

\begin{lem} \label{l:exppoly-geom}
  Let $Q_1$,~$\ldots\,$~$Q_l \in K[\vx]$, let $\vlambda_1$, $\ldots\,$,~$\vlambda_l \in (K^*)^d$ be pairwise distinct, and
  \[
    F = \sum_{\vec n \in \bN^d}\bigg( \sum_{j=1}^l Q_j(\vn) \bm\lambda_j^{\vn} \bigg) \vx^{\vn}.
  \]
  If $F$ is a Bézivin series, then $Q_1$, $\ldots\,$,~$Q_l$ are constant.
\end{lem}

\begin{proof}
  By Corollary~\ref{c:poly-exp-repr}, we have $F=P/Q$ with $Q=(1-\alpha_1 y_1) \cdots (1- \alpha_s y_s)$ where $\alpha_1$, $\ldots\,$,~$\alpha_d \in K^*$ and $y_1$, $\ldots\,$,~$y_s \in \{x_1,\ldots,x_d\}$ (repetition is allowed).
  We proceed by contradiction.
  Suppose that one of the polynomials $Q_j$ is not constant.
  Then a factor in $Q$ occurs with multiplicity $e \ge 2$, again by Corollary~\ref{c:poly-exp-repr}.
  Multiplying $F$ by some polynomial, the Bézivin property is preserved, and hence we may assume $F=P/(1-\alpha x_1)^2$ where $P$ is not divisible by $1-\alpha x_1$ (reindexing the variables if necessary).
  
  Through successive polynomial division, write $P = B_2(1-\alpha x_1)^2 + B_1(1-\alpha x_1) + B_0$ with $B_2 \in K[\vx]$ and $B_1$,~$B_0 \in K[x_2,\ldots,x_d]$.
  Since $(1-\alpha x_1)$ does not divide $P$, we have $B_0 \ne 0$.
  Then
  \[
    \frac{P}{Q} = B_2 + \frac{B_1}{(1-\alpha x_1)} + \frac{B_0}{(1-\alpha x_1)^2} = B_2 + \sum_{n=0}^\infty B_1 \alpha^n x_1^n + (n+1)B_0 \alpha^n x_1^n.
  \]
  Let $\vec u \in K[x_2,\ldots,x_d]$ be a monomial in the support of $B_1$.
  For all sufficiently large $n$,
  \[
    \alpha^{-n} [x_1^{n}\vec u]P = (n+1)[\vec u]B_0  + [\vec u]B_1.
  \]
  Let $A\coloneqq (y+1)[\vec u]B_0 + [\vec u]B_1 \in K[y]$.
  By choice of $\vec u$, the polynomial $A$ has degree $1$.
  But enlarging $G$ to ensure $\alpha \in G$, there exists $r \ge 0$ such that $A(n) \in rG_0$ for all sufficiently large $n$.
  This contradicts Lemma~\ref{l:poly-constant}.
\end{proof}

We are now ready to prove the main theorem of this section.

\begin{proof}[Proof of Theorem~\ref{thm:rational}]
  We may assume that $K$ is finitely generated.
  Let $F=\sum_{\vec n \in \bN^d} f(\vec n) {\vec x}^{\vec n}$ be a $D$-finite power series all of whose coefficients are in $rG_0$ for some $r \ge 0$ and a finitely generated subgroup $G \le K^*$.
  Since the sequence $f\colon \bN^d \to K$  is $P$-recursive, there exist $Q_{j,\vec a}$ as in Equation~\eqref{eq:precursive}.
  For every $j \in [1,d]$ there exists $\vec a \in [0,k]^d$ with $Q_{j,\vec a}(0) \ne 0$.
  Let
  \[
    H \coloneqq \sum_{\vec{a_1}, \ldots, \vec{a_d} \in [0,k]^d} Q_{1,\vec a_1}(0)\cdots Q_{d,\vec a_d}(0) {\vec x}^{\vec a_1+\dots+\vec a_d} F.
  \]
  Then $H=PF$ with a nonzero polynomial $P$, and it suffices to show that $H$ is rational.
  Since $H$ is $D$-finite as well, the sequence of coefficients of $H$, denoted by $h$, is $P$-recursive of some size $k'$; without restriction $k' \ge k$.

  For each $j \in [1,d]$ we will now apply Lemma~\ref{l:geom} in the following way.
  We choose $\Omega = \bN^d$ and $\pi\colon \bN^d \to \bN \subseteq K$ to be the projection on the $j$-th coordinate.
  We set $B = [0,k'd]$ and consider the equation
  \[
    0 = \sum_{\vec a \in [0,k]^d} Q_{j,\vec a}(n_j) f(\vec n - \vec a) = \sum_{\vec a \in [0,k]^d} Q_{j,\vec a}(\pi(\vec n)) f_{\vec a}(\vec n),
  \]
  where $f_{\vec a}\colon \bN^d \to G_0$ is defined by $f_{\vec a}(\vec n) = f(\vec n - \vec a)$.
  By assumption on $f$, this equation holds for $\vec n \in \bN_{\ge k'}^d$.
  Enlarging $G=G'$ if necessary, Lemma~\ref{l:geom} implies that, for any choice of $l \in \bN_{\ge 1} \setminus G$, there exists $e_0 \in \bN$ such that
  \[
    \sum_{\vec a \in [0,k]^d} Q_{j,\vec a}(0) f(\vec n - \vec a) = \sum_{\vec a \in [0,k]^d} Q_{j,\vec a}(0) f_{\vec a}(\vec n) = 0 \quad\text{for}\quad\vec n \in \bN_{\ge k'}^{j-1} \times (l^{e_0+\bN}+B) \times \bN_{\ge k'}^{d-j}.
  \]

  Further enlarging $G=G'$ if necessary, we may assume that the same finitely generated group $G$ works for all $j \in [1,d]$.
  Then we can also take the same $l \in \bN_{\ge 1} \setminus G$ for all $j \in [1,d]$, and finally, taking $e_0$ large enough, we can also assume the same $e_0$ works for all $j \in [1,d]$.
  As a further convenience, we may enlarge $e_0$ to ensure $l^{e_0} \ge k'd$.

  For $j \in [1,d]$, define $C_j \coloneqq l^{e_0+\bN}+[k'(d-1),k'd]$.
  Then, for all $j \in [1,d]$ and $\vec n \in \bN_{\ge k'd}^{j-1} \times C_j \times \bN_{\ge k'd}^{d-j}$,
  \[
    \begin{split}
      h(\vec n) &\coloneqq  \sum_{\vec{a_1}, \ldots, \vec{a_d} \in [0,k]^d} \Big( \prod_{i=1}^d Q_{i,\vec a_i}(0) \Big) \,f(\vec n - \vec a_1 - \dots - \vec a_d) \\
      &= \sum_{\vec{a_1}, \dots, \widehat{\vec{ a_j}}, \dots, \vec{a_d} \in [0,k]^d} \Big( \prod_{\substack{i=1\\i \ne j}}^dQ_{i,\vec a_i}(0)  \Big) \sum_{\vec a_j \in [0,k]^d} Q_{j,\vec a_j}(0) \,f\big( (\vec n - \sum_{\substack{i=1\\ i \ne j}}^d \vec a_i) - \vec a_j\big) = 0.
    \end{split}
  \]
  Because $h$ is $P$-recursive of size $k'$, Lemma~\ref{l:zero} implies that $h$ vanishes on $\bN_{\ge m}^d$ with $m = l^{e_0} + k'(d-1)$.

  For a subset $\emptyset \ne I \subseteq [1,d]$ and a tuple $(j_i)_{i\in I} \in [0,m]^{I}$ let
  \[
    u_{I,(j_i)_{i \in I}} \coloneqq \prod_{i \in I} x_i^{j_i} \prod_{i \in [1,d]\setminus I} x_i^{m+1}.
  \]
  We may write
  \[
    H = \sum_{\emptyset \ne I \subseteq [1,d]}  \sum_{\substack{(j_i)_{i \in I}\\ j_i \in [0,m]}} u_{I,(j_i)_{i \in I}}  H_{I,(j_i)_{i \in I}} \big( (x_\mu)_{\mu \in [1,d]\setminus I}\big),
  \]
  with suitable power series $H_{I,(j_i)_{i \in I}}$ in $d-\card{I}$ variables.
  Note that in each summand of the form
  \[
    u_{I,(j_i)_{i \in I}} H_{I,(j_i)_{i\in I}} \big((x_\mu)_{\mu \in [1,d]\setminus I}\big)
  \]
  the exponents of $x_i$ for $i \in I$ are fixed at some $j_i \in [0,m]$, while the exponents of $x_{i}$ are greater than $m$ for $i \in [1,d]\setminus I$.
  It follows that these summands have pairwise disjoint support, and hence the $H_{I,(j_i)_{i \in I}}$ have their coefficients in $rG_0$.
  These coefficient sequences being shifted sections of $h$, moreover the $H_{I,(j_i)_{i \in I}}$ are $D$-finite.
  Because $I \ne \emptyset$, and by induction on $d$, we may assume that the $H_{I,(j_i)_{i \in I}}$ are rational, and therefore so is $H$.
\end{proof}

\section{The denominator of rational Bézivin series}

Having shown that $D$-finite Bézivin series are rational, we now show that the denominator takes a particularly simple form.
In this section we work with the Hahn series ring $K\llparen x^H \rrparen$ and its additive valuation $\val \colon K\llparen x^H \rrparen \to H \cup \{\infty\}$ (see Subsection~\ref{ssec:hahn}). At first, the group $H$ will be an arbitrary totally ordered abelian group, but we restrict to $H=\bQ^d$ after Proposition~\ref{p:series-is-monomial}.

We make use of the following easy fact.

\begin{lem} \label{l:binom-independent}
  The polynomials
  \[
    \bigg\{\, { x \choose l } = \frac{x (x-1) \cdots (x-l+1)}{l!} \,:\, l \ge 0 \,\bigg\}
  \]
  form a basis of the polynomial ring $K[x]$.
\end{lem}

The following is a crucial reduction step, used later in describing the denominators of rational Bézivin series.
 
\begin{prop} \label{p:series-is-monomial}
  Let $H$ be a totally-ordered abelian group and $L=K \llparen x^H \rrparen$ its Hahn power series ring.
  Let $s \ge 1$ and, for $i \in [1,s]$, let $\alpha_i$,~$\beta_i \in L^*$ with the $\alpha_i$ pairwise distinct.
  For $n \ge 0$, let
  \[
    F_n = \sum_{i=1}^s \beta_i \alpha_i^n \in L.
  \]
  If there exist a finitely generated subgroup $G \le K^*$ and $r \ge 0$ such that for all $n \ge 0$ every coefficient of $F_n$ is an element of $rG_0$, then there exists an $i \in [1,s]$ such that the support of $\alpha_i$ is a monomial, that is
  \[
    \alpha_i = c_i x^{h_i} \qquad\text{with}\qquad c_i \in K^*\text{ and } h_i \in H.
  \]
\end{prop}

\begin{proof}
  Let $m = \min\{ \val(\alpha_1), \ldots, \val(\alpha_s) \}$ and $m'= \min\{\val(\beta_1),\ldots,\val(\beta_s)\}$.
  We may without restriction replace $F_n$ by $x^{-m'-nm} F_n$, and may thus assume $m=m'=0$.
  After reindexing,
  \[
    0 = \val(\alpha_1)=\cdots =\val(\alpha_k) < \val(\alpha_{k+1}) \le \cdots \le \val(\alpha_s) \qquad\text{for some $k \in [1,s]$.}
  \]
  For $i \in [1,k]$, we have $\alpha_i = c_i + \theta_i$ with $c_i \in K^*$ and $\theta_i \in L$ where $\val(\theta_i) > 0$.
  We proceed with a proof by contradiction, and may therefore assume $\theta_i \ne 0$ for all $i \in [1,k]$.
  
  Reindexing the $\alpha_1$, $\ldots\,$,~$\alpha_k$ again if necessary, we may partition
  \[
    [1,k]= [k_1,k_2-1] \uplus [k_2,k_3-1] \uplus \cdots \uplus [k_{p},k_{p+1}-1],
  \]
  with $1=k_1 < k_2 < \cdots < k_{p} < k_{p+1}=k+1$, such that $c_i=c_j$ if and only if $i$,~$j \in [k_{\nu},k_{\nu+1}-1]$ for some $\nu$.
  
  Let
  \[
      T_n \coloneqq \sum_{i=1}^k \beta_i (\alpha_i^n -  c_i^n)  =
       \sum_{\nu=1}^p \bigg[ \sum_{j=k_\nu}^{k_{\nu+1}-1}  \beta_j \alpha_j^n  - \Big( \sum_{j=k_\nu}^{k_{\nu+1}-1} \beta_j \Big) c_{k_\nu}^n \bigg].
  \]
  Then $\val(T_n) > 0$ for all $n$, because the constant terms cancel.

  \medskip
  \noindent
  \textbf{Step 1.}
  \emph{There exists $m \in H$ with $m > 0$ such that $\val(T_n) \le m$ for infinitely many $n \ge 0$.}
  \medskip
  
  Let
  \[
    A_n =
    \begin{pmatrix}
      \alpha_1^n & \cdots & \alpha_k^n & c_{k_{1}}^n & \cdots & c_{k_p}^{n} \\
      \vdots & \ddots & \vdots & \vdots & \ddots & \vdots \\
      \alpha_1^{n+k-1+p} & \cdots & \alpha_k^{n+k-1+p} & c_{k_{1}}^{n+k-1+p} & \cdots & c_{k_p}^{n+k-1+p} \\
    \end{pmatrix}.
  \]
  Extracting scalars, these are Vandermonde matrices with
  \[
    \det(A_n) = \alpha_1^n \cdots \alpha_k^n c_{k_{1}}^n \cdots c_{k_p}^n \prod_{1 \le i < j \le k} (\alpha_j - \alpha_i) \prod_{1 \le i <j \le p} (c_{k_j} - c_{k_i}) \prod_{\substack{1 \le i \le k \\ 1 \le j \le p}} (c_{k_j} - \alpha_i).
  \]
  By construction $\det(A_n) \ne 0$.
  Since $\val(\alpha_i)=0$ for $i \in [1,k]$, moreover $\val(\det(A_n))$ is in fact independent of $n$.
  Let $m_1 \coloneqq \val(\det(A_n))$ and $\vec w = (w_1,\,\ldots,w_{k+p})^T$ with
  \[
    w_i =
    \begin{cases}
      
      \beta_i & \text{if $1 \le i \le k$, and} \\
      -\sum_{j=k_\nu}^{k_{\nu+1}-1} \beta_j & \text{if $i=k+\nu$ with $1 \le \nu \le p$.}
    \end{cases}
  \]

  Since $\beta_1 \ne 0$, we have $w_1 \ne 0$.
  Let $m_2 \coloneqq \val(w_{1})$.
  The $i$-th entry of $A_n\vec w$ is $T_{n+i-1}$.
  Hence, using the subscript $1$ to denote the first coordinate of a vector,
  \[
    \det(A_n)w_{1} =  \big( \operatorname{adj}(A_n)A_n  \vec w )_{1} = \big(\operatorname{adj}(A_n) (T_n, \ldots, T_{n+k+p-1})^T \big)_{1} = \sum_{j=0}^{k+p-1} \gamma_j T_{n+j}
  \]
  for some $\gamma_j$ with $\val(\gamma_j) \ge 0$.
  Since $\val(\det(A_n)w_{1}) = m_1 + m_2$, we must have $\val(T_{n+\delta(n)}) \le m_1 + m_2$ for some $\delta(n) \in [0,k+p-1]$.

  \medskip
  \noindent
  \textbf{Step 2.} \emph{There exist $N_0 \ge 0$ such that for all $n \ge N_0$,
  \[
    \val\bigg( T_n -  \sum_{i=1}^k \sum_{l=1}^{N_0}  {n \choose l} \beta_i c_i^{n-l} \theta_i^l \bigg) > m.
  \]}
\medskip
Substituting $\alpha_i=c_i + \theta_i$ into the definition of $T_n$ and expanding, for $0 \le N_0 \le n$,
\[
  T_n = \sum_{i=1}^k \sum_{l=1}^{N_0} {n \choose l} \beta_i c_i^{n-l} \theta_i^l + \sum_{i=1}^k \sum_{l=N_0+1}^n {n \choose l} \beta_i c_i^{n-l} \theta_i^l.
\]
  There exists $N_0$ such that $\val(\beta_i\theta_i^l) =l\val(\theta_i)+\val(\beta_i) > m$ for all $l \ge N_0$, and choosing such $N_0$ establishes the claim of \textbf{Step 2}.

  \bigskip
  Enlarging $N_0$ if necessary, we may also assume $\val(\alpha_i^{N_0}) > m$ for $i \in [k+1,s]$ by Lemma~\ref{lem:order}.
  Let
  \[
    P_n \coloneqq \sum_{i=1}^k \sum_{l=1}^{N_0}  {n \choose l} \beta_i c_i^{n-l} \theta_i^l.
  \]
  Fix $h \in H$ with $0 < h \le m$ such that $[x^h]P_n \ne 0$ for some $n \ge N_0$.
  The --- crucial --- existence of such an $h$ is guaranteed by Steps 1 and 2.
  Substituting Hahn series expressions for $\beta_i$ and $\theta_i$, we can express $[x^h]P_n$ as
  \[
    [x^h]P_n = \sum_{i \in I} p_{h,i}(n) c_i^n,
  \]
  where $I \subseteq [1,k]$ is such that the $c_i$ are pairwise distinct, and $p_{h,i} \in K[x]$ are polynomials.
  Since $l$ is always at least $1$ in the expression for $P_n$, the polynomials $p_{h,i}$ are either zero or non-constant as a consequence of Lemma~\ref{l:binom-independent}.
  By choice of $h$, there must be at least one $i$ with $p_{h,i} \ne 0$.

  For $n \ge N_0$,
  \[
    [x^h]P_n = [x^h]T_n = [x^h]F_n - \sum_{i=1}^k [x^h]\beta_i c_i^n - \underbrace{[x^h] \sum_{i=k+1}^s \beta_i \alpha_i^n}_{=0}.
  \]
  Enlarging $G$, we may assume $c_i$, $-[x^h]\beta_i \in G$, and hence $[x^h]P_n \in (r+k)G_0$.
  Now the univariate case of Lemma~\ref{l:exppoly-geom} implies that the $p_{h,i}$ are constant, and by our construction therefore $p_{h,i}=0$ for all $i \in I$.
  This is a contradiction to $[x^h]P_n \ne 0$ for some $n \ge N_0$.
\end{proof}

\begin{lem}  \label{l:monomial-exponents}
  Let $K$ be algebraically closed and let
  \[
    P = (1-c_1\vx^{\vec a_1}) \cdots (1-c_s \vx^{\vec a_s}) \in K \llparen x_1^{\bQ}, \ldots, x_d^{\bQ} \rrparen.
  \]
  with $c_i \in K^*$ and $\vec a_1$, $\ldots\,$,~$\vec a_s \in \bQ^d\setminus \{\vec 0\}$.
  \begin{itemize}
  \item 
    If $P \in K[\vx^{\pm 1}]$, then there exist $\vec b_1$, $\ldots\,$,~$\vec b_t \in \bZ^d \setminus \{\vec 0\}$ and $c_1'$, ~$\ldots\,$,~$c_t' \in K^*$ such that
    \[
      P = (1-c_1'\vx^{\vec b_1})\cdots (1-c_t'\vx^{\vec b_t}).
    \]
  \item
    If $P \in K[\vx]$ with $P(\vec 0)=1$, one can even take $\vec b_1$, $\ldots\,$,~$\vec b_t \in \bN^d \setminus \{\vec 0\}$.
  \end{itemize}
\end{lem}

\begin{proof}
  Let $\vec a_i = (a_{i,1},\ldots,a_{i,d})$ for $i \in [1,s]$.
  Let $N \in \bN$ be such that $N a_{i,j} \in \bZ$ for all $i \in [1,s]$ and $j \in [1,d]$.
  Since
  \[
    1 - c_i^N\vx^{N\vec a_i} = (1-c_i \vx^{\vec a_i}) \sum_{j=0}^{N-1} c_i^j \vx^{j\vec a_i},
  \]
  there exists $F \in K \llparen x_1^{\bQ},\ldots,x_d^{\bQ} \rrparen$, having finite support, such that $PF=Q$ with $Q=(1- c_1^N \vx^{N\vec a_1}) \cdots (1- c_s^{N} \vx^{N \vec a_s})$.
  Since $P$, $Q$ are Laurent polynomials, the quotient $F=Q/P$ is a rational function in $\vx$.
  Therefore $Q/P$ has a Laurent series expansion at the origin, which coincides with the Hahn series $F$.
  Thus $F$ is in fact a Laurent series in $\vx$.
  Since $F$ has finite support, it is even a Laurent polynomial.
  We conclude that $P$ divides $Q$ in the Laurent polynomial ring $K[\vx^{\pm 1}]$.
  
  The Laurent polynomial ring is a factorial domain, arising from $K[\vx]$ by inverting the prime elements $x_1$, $\ldots\,$,~$x_d$.
  By Lemma~\ref{l:special-polys} we know that all factors of $Q$ in $K[\vx^{\pm 1}]$ are --- up to units of $K[\vx^{\pm 1}]$ --- of the form $1-c'\vx^{\vec b}$ with $c' \in K^*$ and $\vec b \in \bZ^d \setminus \{\vec 0\}$.
  Therefore
  \[
    P = (1-c_1'\vx^{\vec b_1})\cdots (1-c_t'\vx^{\vec b_t})
  \]
  with $\vec b_1$, $\ldots\,$,~$\vec b_t \in \bZ^d \setminus \{\vec 0\}$, and this factorization is unique up to order and associativity in $K[\vx^{\pm 1}]$.

  Suppose now that $P$ is a polynomial and $P(\vec 0)=1$.
  Then $P=Q_1\cdots Q_r$ with irreducible polynomials $Q_i \in K[\vx]$.
  Since $P(\vec 0)=1$, we have $Q_1(\vec 0)\cdots Q_r(\vec 0)=1$.
  Replacing each $Q_i$ by $Q_i/Q_i(\vec 0)$, we can therefore assume $Q_i(\vec 0)=1$ for each $i \in [1,r]$.
  In particular, no $Q_i$ is a monomial.
  Therefore, $Q_1$, $\ldots\,$, ~$Q_r$ remain prime elements in the localization $K[\vx^{\pm 1}]$.
  Since $K[\vx^{\pm 1}]$ is factorial, this implies $r=t$ and that, after reindexing the polynomials $Q_i$ if necessary, we may assume that $Q_i$ is associated to $1- c_i'\vx^{\vec b_i}$ for each $i \in [1,r]$.
  Explicitly, this means $Q_i = q_i{\vx^{\vec e_i}}(1-c_i'\vx^{\vec b_i})$ with $\vec e_i \in \bZ^d$ and $q_i \in K^*$. Then $Q_i \in K[\vec x]$ forces $\vec e_i \in \bN^d$, and so there are two possibilities: either $\vec e_i=\vec 0$ and $q_i=1$, in which case necessarily and $\vec b_i \in \bN^d$, or $q_i c_i' \vx^{\vec e_i+\vec b_i}=1$.
  In either case $Q_i$ is of the desired form.
\end{proof}

\begin{lem} \label{l:minpoly}
  Let $Q \in K[x_1,\ldots,x_{d-1},x_d]$ be irreducible with $\deg_{x_d}(Q) \ge 1$, and let $L\coloneqq K \llparen x_1^{\mathbb{Q}},\ldots ,x_{d-1}^{\mathbb{Q}} \rrparen$.
  Let
  \[
    Q = \mu \prod_{i=1}^s (1 - \lambda_i x_d)
  \]
  with $\mu \in L^*$ and $\lambda_1$, $\ldots\,$, $\lambda_s \in L$.
  If some $\lambda_i$ is a monomial in $L$, then all of them are.
\end{lem}

\begin{proof}
  Let $R \coloneqq K[x_1,\ldots,x_{d-1}]$ and $L_0 \coloneqq K(x_1,\ldots,x_{d-1})$.
  Since $Q$ is irreducible in $K[x_1,\ldots,x_d]$ and is not contained in $R$, it is also irreducible as a univariate polynomial in $R[x_d]$.
  By Gauss's Lemma, then $Q$ is irreducible in $L_0[x_d]$.
  Since $L$ is algebraically closed, $Q$ can be expressed as a product of linear factors as above.

  Suppose now without restriction that $\lambda_1^{-1}= a \vx^{\vec e}$ with $a \in K^*$ and $\vec e \in \bQ^{d-1}$.
  Let $N \in \bZ_{\ge 1}$ with $Ne \in \bZ^{d-1}$.
  Then $\lambda_1^{-1}$ is a root of $P = x_d^{N} - a^{N}\vx^{eN} \in L_0[x_d]$.
  But $P= \prod_{j=0}^{N-1} (x_d - a\zeta^j \vx^e)$ in $L[\vx]$, where $\zeta \in L^*$ is a primitive $N$-th root of unity.
  Since all roots of $P$ are monomials, and $Q$ divides $P$ by irreducibility, all roots of $Q \in L[x_d]$ are monomials.
\end{proof}

\begin{thm} \label{t:denominator1}
  Let $K$ be algebraically closed and let $F \in K\llbracket \vx \rrbracket$ be rational, with $F=P/Q$, where $P$,~$Q \in K[\vx]$ are coprime polynomials and $Q(\vec 0)=1$.
  If $F$ is a Bézivin series, then there exist $s \ge 0$, non-constant monomials $u_1$, $\ldots\,$,~$u_s \in K[\vx]$, and $c_1$, $\ldots\,$,~$c_s \in K^*$ such that
  \[
    Q = (1-c_1u_1)\cdots (1-c_su_s).
  \]
\end{thm}

\begin{proof}
  We proceed by induction on $d$.
  If $d=0$, then the claim holds trivially with $s=0$.
  Suppose $d \ge 1$ and that the claim holds for $d-1$.

  Let $Q = Q_1 \cdots Q_n$ where $Q_1$, $\ldots\,$,~$Q_n \in K[\vx]$ are irreducible ($n \ge 0$).
  For each $j \in [1,n]$ the series $P/Q_j = FQ_1\cdots Q_{j-1}Q_{j+1}\cdots Q_n$ is also a rational Bézivin series.
  It therefore suffices to show: if $F=P/Q$ with an irreducible polynomial $Q \in K[\vx]$, not dividing $P$, and $Q(\vec 0)=1$ and $F$ is a Bézivin series, then $Q$ is of the form $1-cu$ with $c \in K^*$ and a monomial $u \in K[\vx]$.

  By reindexing, we may assume that $Q$ has degree at least one in the variable $x_d$.
  Then, since $Q$ is irreducible and does not divide $P$, there are polynomials $A$ and $B$ in $K[\vx]$ such that $C \coloneqq AP+BQ\in K[x_1,\ldots,x_{d-1}]\setminus \{0\}$.
  Since $FA$ and $B$ are both rational Bézivin series, the same is true for $C/Q=FA+B$.
  Now since $Q$ has degree at least one in $x_d$ and since it has distinct roots as a polynomial in $x_d$, we may use the fact that $L\coloneqq K \llparen x_1^{\mathbb{Q}},\ldots ,x_{d-1}^{\mathbb{Q}} \rrparen$ is algebraically closed to factor $Q$ and write
  \[
    C/Q = CQ_0^{-1} (1-\lambda_1 x_d)^{-1}\cdots (1-\lambda_s x_d)^{-1},
  \]
  where $Q_0 \in K[x_1,\ldots ,x_{d-1}]$ and $\lambda_1$, $\ldots\,$,~$\lambda_s\in L^*$ are pairwise distinct.
  Using partial fractions, we may write this as
  \[
    C/Q=\alpha + \sum_{i=1}^s \frac{\beta_i}{1-\lambda_i x_d},
  \]
  where $\alpha$,~$\beta_i\in L$ and the $\beta_i$ are nonzero.
  Then if we expand, we have that for $n\ge 1$, the coefficient of $x_d^n$ in $C/Q$ is 
  \[
    F_n\coloneqq \sum_{i=1}^s \beta_i \lambda_i^n\in L.
  \]
  Applying Proposition~\ref{p:series-is-monomial} to this family of Hahn series, one of the $\lambda_i$ is of the form $\lambda_i=c_i u_i$ with $c_i \in K^*$ and $u_i \in L$ a monomial, that is, of the form $u_i = x_1^{e_{i,1}}\cdots x_{d-1}^{e_{i,d-1}}$ with $e_{i,j} \in \bQ$.
  By Lemma~\ref{l:minpoly}, then all $\lambda_i$ are of the form $\lambda_i=c_iu_i$ with $c_i \in K^*$ and $u_i \in L$ a monomial.

  Since $CQ_0^{-1} = CQ^{-1}(1-c_1u_1x_d)\cdots (1-c_su_sx_d)$, it follows that the rational series $CQ_0^{-1}$ is a Bézivin series.
  The induction hypothesis implies $Q_0=(1-c_{s+1}u_{s+1})\cdots (1-c_{t}u_t)$ with $c_j \in K^*$ and monomials $u_{s+1}$,~$\ldots\,$,~$u_t \in K[x_1,\ldots,x_{d-1}]$.
  Replacing $u_i$ by $u_ix_d$ for $i \in [1,s]$, we have
  \[
    Q = (1-c_1u_1) \cdots (1-c_t u_t) \in K[\vx].
  \]
  By Lemma~\ref{l:monomial-exponents}, we can rewrite this product in such a way that the Hahn monomials $u_j$ are monomials with nonnegative integer exponents, that is, monomials in the polynomial ring $K[\vx]$.
  Then $t=1$ by irreducibility of $Q$ and the claim is shown.
\end{proof}


\section{Structural decomposition of rational Bézivin series}

\begin{prop} \label{p:skew-geom}
  Let $K$ be algebraically closed.
  Every rational Bézivin series over $K$ is expressible as a \textup{(}trivially ambiguous\textup{)} sum of skew-geometric rational series.
\end{prop}

\begin{proof}
  By Theorems~\ref{t:denominator1} and \ref{t:polyexp}, the coefficient sequence of $F$ is piecewise polynomial-exponential on simple linear subsets of $\bN^d$.
  Let $\cS = \vec a + \vec b_1 \bN + \cdots + \vec b_s \bN$ be such a simple linear set on which $F$ is polynomial-exponential and set $u_i \coloneqq \vx^{\vec b_i}$ for $i \in [1,s]$ and $\vec u = (u_1,\ldots,u_s)$.
  By Lemma~\ref{l:equiv-pexp}, there exist $Q_1$, $\ldots\,$,~$Q_l \in K[y_1,\ldots,y_s]$ and $\vlambda_1$, $\ldots\,$,~$\vlambda_s \in (K^*)^s$ such that
  \[
    F \odot \ind_\cS = \vx^{\vec a} \sum_{\vec m \in \bN^s} \sum_{j=1}^l Q_j(\vec m) \vlambda_j^{\vec m} {\vec u}^{\vec m}.
  \]
  We consider $H(\vec u) \coloneqq \vx^{-\vec a} (F\odot \ind_\cS)$ as a power series in $\vec u$.
  Clearly $H$ is a Bézivin series.
  By Lemma~\ref{l:exppoly-geom}, the series $H(\vec u)$ is a sum of geometric series in $\vec u$, and thus $F \odot \ind_\cS$ is a sum of skew-geometric series in $\vec x$.
  Since we can partition $\bN^d$ into finitely many disjoint such simple linear sets $\cS$, the series $F$ also is a sum of skew-geometric series.
\end{proof}

\begin{remark}
  Instead of using Theorem~\ref{t:polyexp}, the main reduction in the proof of the previous theorem can also be derived from the fact that the coefficient sequence of a generating series
  \[
    \frac{1}{(1-v_1)\cdots(1-v_l)}
  \]
  with monomials $v_1$, $\ldots\,$,~$v_l$ is piecewise polynomial on simple linear subsets of $\bN^d$.
  Therefore one may substitute results on vector partitions \cite{sturmfels95} for Theorem~\ref{t:polyexp}.
  (In contrast to Theorem~\ref{t:polyexp}, here the coefficients in front of the monomials are all equal to $1$.)
  
  We sketch the main part of the argument.
  Consider an expression of the form $1/Q$ with
  \[
    Q \coloneqq (1-c_1u_1)\cdots (1-c_lu_l),
  \]
  where $c_1$, $\ldots\,$,~$c_l \in K^*$ and $u_1$, $\ldots\,$,~$u_l$ are non-constant monomials.
  First we may assume that the monomials $u_1$, $\ldots\,$,~$u_l$ have a common zero $(\alpha_1,\ldots,\alpha_d)$ in the algebraic closure $\overline{K}$---otherwise we may use Hilbert's Nullstellensatz to reduce the expression $1/Q$ into a sum of expressions with fewer factors in the denominator. (This is actually also the first step of the partial fraction decomposition \cite{raichev12} used in the proof of Theorem~\ref{t:polyexp}.)

  Now let $v$ be a monomial.
  The coefficient of $v$ in $1/Q$ is
  \[
    \begin{split}
      \sum_{\substack{e_1,\ldots, e_{l} \ge 0 \\ u_{1}^{e_1}\cdots u_{l}^{e_{l}}=v}} \prod_{j=1}^{l} c_{j}^{e_j} & = \sum_{\substack{e_1,\ldots, e_{l} \ge 0 \\ u_{1}^{e_1}\cdots u_{l}^{e_{l}}=v}} \prod_{j=1}^{l} u_{j}(\alpha_{1},\ldots,\alpha_{d})^{-e_j}
      \\ &= \sum_{\substack{e_1,\ldots, e_{l} \ge 0 \\ u_{1}^{e_1}\cdots u_{l}^{e_{l}}=v}} v(\alpha_{1},\ldots,\alpha_{d})^{-1} = \mu(v)\,  v(\alpha_{1},\ldots,\alpha_{d})^{-1}.
      \end{split}
  \]
  Here $\mu(v) \in \bN$ is the number of ways of writing $v$ as a product of $u_{1}$, $\ldots\,$,~$u_{l}$.
  Thus $1/Q$ decomposes as a Hadamard product
  \[
    1/Q = \frac{1}{(1-u_{1}) \cdots (1-u_{l})} \,\odot\, \frac{1}{(1 - \alpha_{1}^{-1}x_1)\cdots (1- \alpha_{d}^{-1}x_d)}.
  \]
  The coefficients of the left factor are therefore polynomial on simple linear subsets of $\bN^d$ and one proceeds from there.
\end{remark}

\subsection{The constants can be taken in the group}

Let $\valpha=(\alpha_1,\ldots,\alpha_s), \vbeta=(\beta_1,\ldots,\beta_s) \in (K^*)^s$ with $s \in \bN$.
We say that $\valpha$ and $\vbeta$ are \defit{relatively non-torsion} if none of the quotients $\alpha_i/\beta_i$ for $i \in [1,s]$ is a root of unity.
Let $u_1$, $\ldots\,$,~$u_s \in K[\vec x]$ be algebraically independent monomials.
Consider an expression of the form
\[
  \frac{P}{(1- \alpha_1 u_1)\cdots (1-\alpha_s u_s)} + \frac{Q}{(1- \beta_1 u_1)\cdots (1-\beta_s u_s)}
\]
with polynomials $P$,~$Q \in K[\vx]$.
If, say, $\alpha_1/\beta_1$ is a root of unity of order $n$, then we may use the identity
\[
  (1-\alpha_1^n u_1^n)= (1-\alpha_1u_1)\bigg( \sum_{j=0}^{n-1}\alpha_1^ju_1^j \bigg).
\]
to replace $1- \alpha_1u_1$ in the denominator by $1-\alpha_1^nu_1^n$ and analogously for $1-\beta_1 u_1$.
We may thus assume that in any representation, the coefficient vectors in the denominator are relatively non-torsion (also for more than two summands).
In regards to the coefficient sequence, and its description in terms of simple linear sets, this amounts to a refinement
\[
  \vec a + \vec b_1 \bN + \cdots + \vec b_s \bN = \bigcup_{j \in [0,n-1]^d} (\vec a + j \vec e_1) + n \vec b_1\bN + \vec b_2 \bN \cdots + \vec b_s \bN.
\]

\begin{lem} \label{l:lpower}
  Let $\vlambda=(\lambda_1,\ldots,\lambda_d) \in (K^*)^d$.
  \begin{enumerate}
  \item\label{lpower:inf} If some $\lambda_i$ is not a root of unity, then $\{\, \vlambda^{\vec n} : \vec n \in \bN^d \} \subseteq K^*$ is infinite.
  \item\label{lpower:linear} For every $c \in K^*$, the set $\mathcal L\coloneqq \{\, \vec n \in \bN^d : \vlambda^{\vec n} = c \,\}$ is a semilinear set.
    If moreover $\{\, \vlambda^{\vec n} : \vec n \in \bN^d \,\}$ is infinite, then $\mathcal L$ has rank at most $d-1$.
  \end{enumerate}
\end{lem}

\begin{proof}
  \ref{lpower:inf} Clear.

  \ref{lpower:linear}
  We may assume $\mathcal L \ne \emptyset$.
  First observe that $U\coloneqq \{\, \vec n \in \bZ^d : \vlambda^{\vec n} = 1 \,\}$ is a subgroup of $\bZ^d$ and therefore free abelian.
  If the set $\{\, \vlambda^{\vec n} : \vec n \in \bN^d \,\}$ is infinite, then $\bZ^d/U$ must be infinite, and therefore $U$ has rank at most $d-1$.
  Now $\mathcal L_0 \coloneqq U \cap \bN^d$ is a linear set.
  
  By Dickson's Lemma (see \cite[Lemma II.7.1]{sakarovitch09}), the set $\mathcal L$ has finitely many minimal elements $\vec a_1$, $\ldots\,$,~$\vec a_k$ with respect to the partial order on $\bN^d$, and it follows  that $\mathcal L = \bigcup_{i=1}^k \vec a_i + \mathcal L_0$.
\end{proof}

\begin{prop} \label{p:coeff-in-g}
  Let $F$ be a trivially ambiguous sum of skew-geometric series that is Bézevin and, more specifically, has coefficients in $rG_0$ for some $r \ge 0$.
  Then $F$ can be written as an \emph{unambiguous sum} of series of the form
  \[
    F_{\cS}=\sum_{i=1}^l \frac{g_{i,0}u_0}{(1-g_{i,1}u_1)\cdots (1-g_{i,s}u_s)},
  \]
  where $u_0$, $u_1$, $\ldots\,$,~$u_s \in K[\vx]$ are monomials with $u_1$, $\ldots\,$,~$u_s$ algebraically independent, and where $g_{i,\nu} \in G$ for all $i \in [1,l]$ and $\nu \in [1,s]$.
  Moreover, one can take $l \le r$.
\end{prop}

\begin{proof}
  We may partition $\bN^d$ into simple linear sets so that on each such simple linear set $\cS$,
  \begin{equation} \label{eq:fsumref}
    F \odot \ind_\cS = F_\cS = \sum_{i=1}^l \frac{c_{i,0}u_0}{(1-c_{i,1}u_1)\cdots (1-c_{i,s}u_s)},
  \end{equation}
  with $c_{i,\nu} \in K^*$, with algebraically independent monomials $u_1,\ldots,u_s$, and with an arbitrary monomial $u_0$.
  We first show that the coefficients in the denominators can be taken in $G$.
  After that we will deal with the numerators and show that we can also achieve $l \le r$.

  \smallskip
  First, by combining summands with the same denominator, we may assume that no denominator occurs twice in \eqref{eq:fsumref}.
  Then the uniqueness statement of Corollary~\ref{c:poly-exp-repr}, applied over the polynomial ring $K[u_1,\ldots,u_s]$, implies that the representation of $F_\cS$ in this form is unique.
  In particular, each $1-c_{i,\nu}u_\nu$ indeed occurs as a factor of the reduced denominator of $F_\cS$, considered as a rational function in $K(u_1,\ldots,u_s)$.

  To deal with the denominators, it suffices to show that for every $i \in [1,l]$ and $\nu \in [1,s]$ there exists $N \ge 1$ such that $c_{i,\nu}^N \in G$.
  Then we can use the remarks preceding this proposition to replace the denominators in such a way that the coefficients are in $G$.
  By symmetry, it suffices to show this for $i=\nu=1$.
  Multiplying $u_0^{-1}F_\cS$ by a suitable polynomial in $K[u_1,\ldots,u_s]$ to clear all denominators other than $1-c_{1,1}u_1$ and setting $c \coloneqq c_{1,1}$, we obtain a rational series
  \[
    H\coloneqq \frac{P(u_1,\ldots,u_s)}{1-cu_1},
  \]
  with $1-cu_1$ not dividing $P(u_1,\ldots,u_s) \in K[u_1,\ldots,u_s]$ (note that the elements $1-c_{i,\nu}u_\nu$ are irreducible in $K[u_1,\ldots,u_s]$).
  By Lemma~\ref{l:multiply}, there is a finite set $B$ such that all coefficients of $H$ are contained in $BG_0$.
  
  Working now in the polynomial ring $K[u_2,\ldots,u_s][u_1]$, we can use polynomial division to write
  \[
    P=Q(1-cu_1) + R
  \]
  with $Q \in K[u_1,u_2,\ldots,u_s]$ and $0 \ne R \in K[u_2,\ldots,u_s]$.
  Fix any monomial $v \in \supp(R)$ and let its coefficient in $R$ be $a$.
  Then, for large $n$,
  \[
    [vu_1^n]H = a c^n \in BG_0
  \]
  For $b \in B$ let $g_b \colon \bN \to G_0$ be such that 
  \[
    a c^n = \sum_{b \in B} b g_b(n) \qquad\text{for large $n$}.
  \]
  Fix a subset $B' \subseteq B$ such that, for infinitely many $n$,
  \[
    a c^n - \sum_{b \in B'} b g_b(n) = 0
  \]
  yields a non-degenerate solution to the unit equation $X_0 + \sum_{b \in B'} X_b=0$ over the group $G'$ generated by $G$, $B'$, and $-1$.
  Let $b \in B'$.
  Then there exists a constant $\beta \in K^*$ such that
  \[
    \frac{a c^{n}}{b g_b(n)} = \beta,
  \]
  for infinitely many $n$.
  Let $n_1 < n_2$ be two such values.
  Then
  \[
    \frac{a c^{n_1}}{b g_b(n_1)} = \frac{a c^{n_2}}{b g_b(n_2)}
  \]
  implies $c^{n_2-n_1} \in G$.
  This finishes the claim about the denominators.
  
  \medskip
  Going back to the representation of $F_\cS$ in \eqref{eq:fsumref}, and refining the set $\cS$ if necessary, we can therefore assume $c_{i,\nu} \in G$ for all $i \in [1,l]$ and $\nu \in [1,s]$.
  Using the same type of reduction, we may assume that whenever $i$, $j \in [1,l]$ and $\nu \in [1,s]$ are such that $c_{i,\nu} \ne c_{j,\nu}$, then $c_{i,\nu}/c_{j,\nu}$ is not a root of unity.

  \medskip
  Now we deal with the numerators $c_{0,\nu}$ for $\nu \in [1,s]$ and simultaneously with the number of summands.
  Let $\vec c_i = (c_{i,1},\ldots,c_{i,s})$ for $i \in [1,l]$, and $\vec u = (u_1,\ldots,u_s)$.
  Then
  \[
    [u_0\vec u^{\vec n}] F = \sum_{i=1}^l c_{i,0} \vec c_i^{\vec n} = g_1(\vec n) + \cdots + g_r(\vec n),
  \]
  for some functions $g_1$, $\ldots\,$,~$g_r \colon \bN^d \to G_0$.
  Consider this as a unit equation over the group generated by $G$, $-1$, and $c_{i,\nu}$ with $i \in [1,l]$ and $\nu \in [0,s]$.

  Consider a partition $\cP=\{ (I_1,J_1), \ldots, (I_t,J_t) \}$ of the disjoint union $[1,l] \uplus [1,r]$, that is $I_\tau \subseteq [1,l]$, $J_\tau \subseteq [1,r]$, at least one of $I_\tau$ and $J_\tau$ is non-empty, and $I_\tau \cap I_{\tau'} = \emptyset = J_\tau \cap J_{\tau'}$ for $\tau \ne \tau'$.
  Let $\Omega_{\cP} \subseteq \bN^s$ denote the set of all $\vec n \in \bN^s$ such that
  \[
    \sum_{i \in I_\tau} c_{i,0} \vec c_i^{\vec n} - \sum_{j \in J_\tau} g_j(\vec n) =0
  \]
  yields a non-degenerate solution to the unit equation $\sum_{i \in I_{\tau}} X_i + \sum_{j \in J_\tau} X_j = 0$.
  Then the sets $\Omega_\cP$ cover $\bN^s$ as $\cP$ ranges through all partitions.

  We claim that there exists a partition $\cP$ such that for all $i \ne j \in [1,l]$ the quotient $(\vec c_{i}/\vec c_j)^{\vec n}$ takes infinitely many values as $\vec n$ ranges through $\Omega_\cP$.
  Suppose $\beta \in K^*$.
  The set of all $\vec n \in \bN^s$ such that $(\vec c_i/\vec c_j)^{\vec n} = \beta$ is a semilinear set of rank at most $s-1$ by Lemma~\ref{l:lpower}.
  Hence, if $(\vec c_i/\vec c_j)^{\vec n}$ takes finitely many values on $\Omega_\cP$, then $\Omega_\cP$ is contained in a semilinear set of rank at most $s-1$.
  However, since the $\Omega_\cP$ cover $\bN^s$, there must be at least one $\cP$ such that $(\vec c_i/\vec c_j)^{\vec n}$ takes infinitely many values for $\vec n \in \Omega_\cP$.
  
  Now fix such a partition $\cP$.
  Then necessarily $\card{I_\tau} \le 1$ for all $\tau \in [1,t]$ by the theorem on unit equations.
  If $I_\tau = \{i\}$, it follows that
  \[
    c_{i,0} = \sum_{j \in J_\tau} g_j(\vec n)/\vec c_i^{\vec n} \qquad\text{for $\vec n \in \Omega_\cP$.}
  \]
  Thus $c_{i,0}$ is a sum of at most $\card{J_\tau}$ elements of $G$.
  Substituting into \eqref{eq:fsumref} and splitting the sums accordingly (now we allow the same denominator to appear multiple times), we achieve $c_{i,0} \in G$ and $l \le \card{J_1} + \cdots + \card{J_t} \le r$ in this representation.
\end{proof}

\section{Proofs of main theorems}

At this point, Theorems~\ref{thm:main-bezivin} and \ref{thm:main-polya} can easily be proven as follows.

\begin{proof}[Proof of Theorem~\ref{thm:main-bezivin}]
  \ref{t-bezivin:dfinite}$\,\Rightarrow\,$\ref{t-bezivin:rational}
  By Theorem~\ref{thm:rational}, every $D$-finite Bézivin series is rational.

  \ref{t-bezivin:rational}$\,\Rightarrow\,$\ref{t-bezivin:skewgeom-refined}
  Let $F \in K\llbracket \vx\rrbracket$ be a rational Bézivin series.
  By Proposition~\ref{p:skew-geom}, the series is a trivially ambiguous sum of skew-geometric series.
  That the constants in the skew-geometric summands can be taken in $G$, and that the sum can be taken in such a way that no $\vec n \in \bN^d$ is contained in the support of more than $r$ summands, follows from Proposition~\ref{p:coeff-in-g}.

  \ref{t-bezivin:skewgeom-refined}$\,\Rightarrow\,$\ref{t-bezivin:skewgeom} Clear.

  \ref{t-bezivin:skewgeom}$\,\Rightarrow\,$\ref{t-bezivin:dfinite} Trivial, because every rational series is $D$-finite.

  \ref{t-bezivin:skewgeom-refined}$\,\Leftrightarrow\,$\ref{t-bezivin:coeffs} This equivalence is immediate from the definition of a skew-geometric series, Definition~\ref{d:skew-geom}, together with Lemma~\ref{l:independence}.
\end{proof}

\begin{proof}[Proof of Theorem~\ref{thm:main-polya}]
  We apply Theorem~\ref{thm:main-bezivin} with $r=1$.
  Then \ref{t-polya:dfinite}$\,\Leftrightarrow\,$\ref{t-polya:rational} of Theorem~\ref{thm:main-polya} follows from \ref{t-bezivin:dfinite}$\,\Leftrightarrow\,$\ref{t-bezivin:rational} of Theorem~\ref{thm:main-bezivin}.
  The implication \ref{t-polya:rational}$\,\Rightarrow\,$\ref{t-polya:skewgeom} of Theorem~\ref{thm:main-polya} follows from \ref{t-bezivin:rational}$\,\Rightarrow\,$\ref{t-bezivin:skewgeom-refined} of Theorem~\ref{thm:main-bezivin}, taking into account $r=1$.

  To obtain \ref{t-polya:skewgeom}$\,\Rightarrow\,$\ref{t-polya:coeff} of Theorem~\ref{thm:main-polya}, apply \ref{t-bezivin:skewgeom-refined}$\,\Rightarrow\,$\ref{t-bezivin:coeffs} of Theorem~\ref{thm:main-bezivin}, noting again $r=1$.
  This gives a partition of $\bN^d$ into simple linear sets, so that on each such set $\cS = \vec a_0 + \vec a_1 \bN + \cdots + \vec a_s \bN$ one has $f(\vec a_0 + m_1 \vec a_1 + \cdots + m_s \vec a_s) = 0$ or $f(\vec a_0 + m_1 \vec a_1 + \cdots + m_s \vec a_s) = g_0 g_1^{m_1} \cdots g_s^{m_s} \ne 0$ with $g_0$, $g_1$, $\ldots\,$,~$g_s \in G$.
  Taking only these simple linear sets of the partition on which $f$ does not vanish, we obtain a partition of the \emph{support} of $F$, as claimed.

  Finally, the implication \ref{t-polya:coeff}$\,\Rightarrow\,$\ref{t-polya:dfinite} of Theorem~\ref{thm:main-polya} follows from the implication \ref{t-bezivin:coeffs}$\,\Rightarrow\,$\ref{t-bezivin:dfinite} of Theorem~\ref{thm:main-bezivin}.
\end{proof}

Before proving Corollary~\ref{c:hadamard}, we show that algebraic series and their diagonals and sections are finitary $D$-finite.
Recall that the operator $I_{1,2} \colon K\llbracket \vx \rrbracket \to K\llbracket \vx \rrbracket$ assigns to
\[
  F=\sum_{n_1,n_2,\ldots,n_d \in \bN} f(n_1,n_2,n_3,\ldots,n_d) x_1^{n_1}x_2^{n_2}x_3^{n_3}\cdots x_d^{n_d}
\]
its primitive diagonal
\[
  I_{1,2}F = \sum_{n_1,n_3,\ldots n_d \in \bN} f(n_1,n_1,n_3,\ldots,n_d) x_1^{n_1} x_3^{n_3} \cdots x_d^{n_d}.
\]
For $1 \le i < j \le d$, the primitive diagonal operators $I_{i,j}$ are defined analogously.
A \defit{diagonal of $F$} is any composition of diagonal operators applied to $F$.
For instance, the complete diagonal of $F$ is $I_{1,2}I_{2,3}\cdots I_{d-1,d}F = \sum_{n \in \bN} f(n,n,\ldots,n) x_1^n$.
A series $F \in K\llbracket \vx \rrbracket$ is \defit{algebraic} if it is algebraic over the field $K(\vec x)$.

\begin{lem} \label{l:algebraic-finitary}
  If $F \in K\llbracket \vx \rrbracket$ is
  \begin{itemize}
    \item algebraic, or
    \item a diagonal of an algebraic series, or
    \item a section of an algebraic series,
  \end{itemize}
  then $F$ is finitary $D$-finite.
\end{lem}

\begin{proof}
  Algebraic series are $D$-finite, and diagonals and sections of $D$-finite series are $D$-finite \cite[Proposition 2.5 and Theorem 2.7]{lipshitz89}. Moreover, diagonals and sections of finitary series are trivially finitary.
  It therefore suffices to show that algebraic series are finitary.
  
  Let $F \in K\llbracket \vx \rrbracket$ be algebraic. Replacing $F$ by $F - F(0)$, we may assume $F(0)=0$.
  Since $F$ is algebraic, there exist $p_2$, $\ldots\,$,~$p_m$, $q \in K[\vx]$ with $q \ne 0$ such that $F = \frac{p_2}{q} F^2 + \cdots + \frac{p_m}{q} F^m$.
  Setting $G = F/q$ and multiplying the previous equation by $1/q$, we obtain an equation
  \[
    G = p_2 G^2 + p_3q G^3 + \cdots + p_m q^{m-2} G^m.
  \]
  Keeping in mind $G(0)=0$, this yields a recursion for the coefficients of $G$.
  It follows from this recursion that $G$ is finitary.
  Since $F=qG$, also $F$ is finitary.
\end{proof}

\begin{lem} \label{l:hadamard-polya}
  If $F \in K\llbracket \vx \rrbracket$ is a finitary series and $F^\dagger$ is also finitary, then $F$ is a Pólya series.
\end{lem}

\begin{proof}
  Since $F$ is finitary, there is a finitely generated $\bZ$-algebra $A \subseteq K$ containing all coefficients of $K$.
  Since $F^\dagger$ is finitary as well, we may also assume that $A$ contains $f(\vec n)^{-1}$  whenever $f(\vec n) \ne 0$.
  Therefore, all nonzero coefficients of $F$ are contained in the unit group $G$ of $A$.
  By a theorem of Roquette (see \cite[Corollary 7.5 of Chapter 2]{lang83}), the group $G$ is finitely generated.
\end{proof}

\begin{proof}[Proof of Corollary~\ref{c:hadamard}]
  Since $F$ and $F^\dagger$ are finitary, $F$ is a Pólya series by Lemma~\ref{l:hadamard-polya}.
  Being also $D$-finite, $F$ satisfies condition \ref{t-polya:dfinite} of Theorem~\ref{thm:main-polya}.
  Conversely, if $F$ satisfies \ref{t-polya:coeff} of Theorem~\ref{thm:main-polya}, then clearly the same is true for $F^\dagger$.
  Thus, the series $F^\dagger$ is $D$-finite.

  The ``in particular'' statement of the corollary follows from Theorem~\ref{thm:main-polya} applied to $F$, respectively, $F^\dagger$.
\end{proof}


\end{document}